\documentclass[article, 11pt,reqno]{amsart}
\usepackage[percent]{overpic}
\usepackage{calc,graphicx,amsfonts,amsthm,amscd,amsmath,amssymb,enumerate,dsfont,mathrsfs,paralist,bbm,tikz}
\makeatletter
\newcommand\makebig[2]{%
  \@xp\newcommand\@xp*\csname#1\endcsname{\bBigg@{#2}}%
  \@xp\newcommand\@xp*\csname#1l\endcsname{\@xp\mathopen\csname#1\endcsname}%
  \@xp\newcommand\@xp*\csname#1r\endcsname{\@xp\mathclose\csname#1\endcsname}%
}
\makeatother

\makebig{biggg} {3.0}
\makebig{Biggg} {3.5}
\makebig{bigggg}{4.0}
\makebig{Bigggg}{4.5}

\usepackage[left=10mm,right=10mm]{geometry}
\usetikzlibrary{patterns,arrows,shapes,backgrounds,shapes,positioning,petri,topaths}
\usetikzlibrary{decorations}
\usetikzlibrary{decorations.pathreplacing}
\usepackage{subfig}
\usepackage{pdfsync}
\usepackage{stmaryrd}
\usepackage[numeric,initials,nobysame]{amsrefs}
\usepackage{marginnote}
\usepackage{hyperref}
\usepackage[bottom]{footmisc}
\usepackage{booktabs}
\newcommand{\ie}{\hbox{\it i.e. }}

\renewcommand{\restriction}{\mathord{\upharpoonright}}

\reversemarginpar
\newlength\fullwidth
\numberwithin{equation}{section}

\DeclareMathSymbol{\leqslant}{\mathalpha}{AMSa}{"36} % nicer `smaller or equal'
\DeclareMathSymbol{\geqslant}{\mathalpha}{AMSa}{"3E} % nicer `larger or equal'
\DeclareMathSymbol{\eset}{\mathalpha}{AMSb}{"3F}     % nicer `emptyset'
\def\<{\langle}
\def\>{\rangle}

\renewcommand{\le}{\;\leqslant\;}                   % redef. of < or =
\renewcommand{\ge}{\;\geqslant\;}   
\renewcommand{\leq}{\;\leqslant\;}                   % redef. of < or =
\renewcommand{\geq}{\;\geqslant\;}                   % redef. of > or =
\newcommand{\maxtwo}[2]{\max_{\substack{#1 \\ #2}}} % max with 2 lines
\newcommand{\sumtwo}[2]{\sum_{\substack{#1 \\ #2}}} % sum with 2 lines
\def\0{\textbf{0}}

\newcommand{\eps}{\epsilon}

\def\1{\ifmmode {1\hskip -3pt \rm{I}} \else {\hbox {$1\hskip -3pt \rm{I}$}}\fi}

\newcommand{\var}{\operatorname{Var}}

\newcommand{\tc}{\thinspace |\thinspace}
\newcommand{\id}{\mathbbm{1}}

\newcommand{\trel}{T_{\rm rel}}

\renewcommand{\l}{\lambda}
\renewcommand{\L}{\Lambda}

\renewcommand{\l}{\lambda}
\renewcommand{\a}{\alpha}

\renewcommand{\t}{\tau}

\newcommand{\g}{\gamma}
\newcommand{\G}{\Gamma}
\newcommand{\z}{\zeta}
\newcommand{\e}{\varepsilon}

\newtheorem{theorem}{Theorem}[section]
\newtheorem*{theorem*}{Theorem}
\newtheorem{lemma}[theorem]{Lemma}
\newtheorem{proposition}[theorem]{Proposition}
\newtheorem{corollary}[theorem]{Corollary}

\newtheorem{claim}[theorem]{Claim}

\newtheorem*{question*}{Question}

\theoremstyle{definition}
\newtheorem{definition}[theorem]{Definition}
\newtheorem*{remark*}{Remark}
\newtheorem*{idefinition*}{Definition}
\newtheorem{remark}[theorem]{Remark}

\newcommand{\cA}{\ensuremath{\mathcal A}}
\newcommand{\cB}{\ensuremath{\mathcal B}}
\newcommand{\cC}{\ensuremath{\mathcal C}}
\newcommand{\cD}{\ensuremath{\mathcal D}}

\newcommand{\cF}{\ensuremath{\mathcal F}}
\newcommand{\cG}{\ensuremath{\mathcal G}}
\newcommand{\cH}{\ensuremath{\mathcal H}}
\newcommand{\cI}{\ensuremath{\mathcal I}}

\newcommand{\cL}{\ensuremath{\mathcal L}}
\newcommand{\cM}{\ensuremath{\mathcal M}}
\newcommand{\cN}{\ensuremath{\mathcal N}}
\newcommand{\cO}{\ensuremath{\mathcal O}}

\newcommand{\cQ}{\ensuremath{\mathcal Q}}

\newcommand{\cS}{\ensuremath{\mathcal S}}

\newcommand{\cU}{\ensuremath{\mathcal U}}

\newcommand{\bbE}{{\ensuremath{\mathbb E}} }

\newcommand{\bbI}{{\ensuremath{\mathbb I}} }

\newcommand{\bbN}{{\ensuremath{\mathbb N}} }

\newcommand{\bbP}{{\ensuremath{\mathbb P}} }

\newcommand{\bbR}{{\ensuremath{\mathbb R}} }

\newcommand{\bbZ}{{\ensuremath{\mathbb Z}} }
\newcommand{\Z}{{\ensuremath{\mathbb Z}} }

\let\a=\alpha      \let\e=\varepsilon
 \let\g=\gamma       \let\l=\lambda
      \let\o=\omega      
   \let\t=\tau   
  \let\z=\zeta
   \let\G=\Gamma  \let\L=\Lambda 
\let\O=\Omega

\renewcommand{\to}{\rightarrow}

\let\epsilon\varepsilon
\let\eps\varepsilon

\def\H{\mathbb{H}}

\def\Z{\mathbb{Z}}

\def\U{\mathcal{U}}

\def\stab{\mathcal{S}}

\def\<{\langle}
\def\>{\rangle}

\def\0{\mathbf{0}}

\begin{document}
\title[]{Exact asymptotics for Duarte and supercritical
  rooted kinetically constrained
  models} 
\author[L. Mar{\^e}ch\'e]{L. Mar{\^e}ch\'e}
\email{mareche@lpsm.paris}
\address{LPSM UMR 8001, Universit\'e Paris Diderot, Sorbonne Paris Cit\'e,
  CNRS, 75013 Paris, France}

\author[F. Martinelli]{F. Martinelli}
\email{martin@mat.uniroma3.it}
\address{Dipartimento di Matematica e Fisica, Universit\`a Roma Tre, Largo S.L. Murialdo 00146, Roma, Italy}

\author[C. Toninelli]{C. Toninelli}
\email{cristina.toninelli@upmc.fr}
\address{ LPSM UMR 8001, Universit\'e Paris Diderot, Sorbonne Paris Cit\'e,
  CNRS, 75013 Paris, France}

\thanks{This work has been supported by the ERC Starting Grant 680275 MALIG. F.M. acknowledges support of PRIN 2015 5PAWZB ''Large Scale Random Structures'' and C.T.  of the ANR-15-CE40-0020-02 grant LSD}



\begin{abstract}
Kinetically constrained models (KCM) are reversible interacting particle systems on $\mathbb Z^d$ with continuous time Markov dynamics of Glauber type, which represent a natural stochastic (and non-monotone) counterpart of the family of cellular automata known as $\mathcal U$-bootstrap percolation. Furthermore, KCM have an interest 
in their own since they display some of the most striking features of the  liquid-glass transition, a major and longstanding open problem in condensed matter physics.
A key issue for KCM is to identify the scaling of the characteristic time scales  when the equilibrium density of empty sites, $q$,  goes to zero. In \cite{MT,MMoT} a general scheme was devised to determine a sharp upper bound for these time scales. Our paper is devoted to developing a (very different) technique which allows to prove matching lower bounds.
We analyse the class of two-dimensional {\sl supercritical rooted KCM} and the {\sl Duarte KCM}, the most studied critical $1$-rooted model. We
 prove that the relaxation time  and the mean infection time diverge for supercritical rooted KCM as $e^{\Theta((\log q)^2)}$ and for Duarte KCM  as $e^{\Theta((\log q)^4/q^2)}$ when $q\downarrow 0$. 
 %club brutto? meglio e^? aggiungo parentesi?
These results prove the conjectures  put forward in 
%\club{we put forward together with R.Morris?}
\cite{Robsurvey,MMoT}, and establish that the time scales for these KCM  diverge much faster than for the
corresponding $\cU$-bootstrap processes, the main reason being the occurrence of energy barriers which determine the dominant
behaviour for KCM, but which do not matter for the  bootstrap dynamics.  
%\club{\bf qui unveiled mi sembra poco adatto, ma vorrei dire che la prova non e' un lower bound cieco in cui esce magicamente il risultato..hai formulazione migliore?}

\end{abstract}
\maketitle
%\tableofcontents

%\club{Title non mi piace exact..}
\section{Introduction}

%club{mi sono attenuta alla tua notazione "Duarte-KCM" ma penso sia meglio cambiare a Duarte KCM senza trattino. non per fare i pignoli, ma e' piu' flessibile, quando tipo voglio dire Duarte U-bootstrap and KCM processes mi funziona meglio senza il trattino..non mi impuntero' ;)}

Kinetically constrained models (KCM) are interacting particle systems on the integer lattice $\mathbb Z^d$, which were introduced in the physics literature in the 1980s in order to model the liquid-glass transition (see e.g.  \cites{Ritort,GarrahanSollichToninelli} for reviews), a major and still largely open problem in condensed matter physics \cite{Berthier-Biroli}. 
%The main motivation for the ongoing (and extremely active) research on KCM is that, despite their simplicity, they feature some of the main signatures of a super-cooled liquid near the glass transition point. 
%also a key tool in the study of jamming transitions towards amorphous solids that occur in a large variety of other systems.
A generic KCM is a continuous time Markov process of Glauber type  characterised by  a finite collection of finite subsets of $\mathbb Z^d\setminus \{ \mathbf{0} \}$, $\,\mathcal U=\{X_1,\dots,X_m\}$, its {\sl update family}. A configuration $\o$ is defined by assigning to 
each site $x\in\mathbb Z^d$ an occupation variable $\omega_x\in\{0,1\},$ corresponding to an empty or occupied site respectively. Each site $x\in\mathbb Z^d$ waits an independent, mean one, exponential time and then, iff  there exists $X\in \cU$ such that $\o_y=0$ for all $y\in X+x$, site $x$ is updated to occupied with probability $p$ and to empty with probability $q=1-p$. 
 %A generic KCM is a continuous time Markov process of Glauber type defined as follows. A configuration $\o$ is defined by assigning to 
%each site $x\in\mathbb Z^d$ an occupation variable $\omega_x\in\{0,1\},$ corresponding to an empty or occupied site respectively. Each site waits an independent, mean one, exponential time and then, iff  a certain local constraint is satisfied by the current configuration $\o$, its occupation variable is updated to occupied with probability $p$ and to empty with probability $q=1-p$. 
%All the constraints that have been considered in the physics literature belong to the following general class \cite{CMRT}.
%Fix an {\sl update family} $\,\mathcal U=\{X_1,\dots,X_m\}$, that is, a finite collection of finite subsets of $\mathbb Z^d\setminus \{ \mathbf{0} \}$. Then $\o$ satisfies the constraint at site $x$ if there exists $X\in \cU$ such that $\o_y=0$ for all $y\in X+x$.
% By choosing appropriately the update family $\cU$, physicists try to mimic the constraints that slow down the motion of a dense liquid and that eventually lead to the dynamical arrest, namely to the formation of the structureless glass state.
Since each {\sl update set} $X_i$ belongs to $\mathbb Z^d\setminus \{ \mathbf{0} \}$,
the constraints never depend on the state of the to-be-updated site. As a consequence, the dynamics satisfies detailed balance w.r.t. the product Bernoulli($p$) measure, $\mu$, which is therefore a reversible invariant measure. Hence, the process started at $\mu$ is stationary.  
%Despite this trivial equilibrium measure,  KCM display an
%extremely rich behaviour which is qualitatively different from that of
%interacting particle systems with non-degenerate birth/death rates
%(e.g. the stochastic Ising model). This behaviour
%    includes the key dynamical features of real glassy materials:
%    anomalously long mixing times~\cites{Aldous,CMRT,MT}, aging and dynamical heterogeneities~\cite{FMRT-cmp}, and ergodicity breaking transitions corresponding to percolation of blocked structures~\cite{GarrahanSollichToninelli}. Moreover, proving the above results rigorously turned out to be a surprisingly challenging task, in part due to the fact that several of the classical tools typically used to analyse reversible interacting particle systems (e.g., coupling, censoring, logarithmic Sobolev inequalities) fail for KCM.

Both from  a physical and from a mathematical point of view, a central issue for KCM is to determine the speed of divergence of the characteristic time scales when $q\downarrow 0$. Two key quantities are: (i) the {\sl relaxation time} $\trel$, i.e. the inverse of the spectral gap of the Markov generator and (ii) the {\sl mean infection time} $\mathbb E_{\mu}(\tau_0)$, i.e. the mean over the stationary process of the first time at which the origin becomes empty. 
The study of the infection time has been  largely addressed   for the $\cU$-bootstrap percolation \cite{BDMS,BSU,BBPS}, a class of discrete cellular automata that can be viewed  as the monotone deterministic counterpart of KCM.
For the $\mathcal U$-bootstrap, given a set of "infected" sites $A_t\subset\mathbb Z^d$ at time $t$, infected sites remain infected, and a site $x$ becomes infected at time $t + 1$ if the translate by $x$ of one of the update sets in $\cU$ belongs to $A_t$. Thus, if infected  (non infected) sites are regarded as empty (respectively occupied) sites, the constraint that has to be satisfied  to infect a site for the $\mathcal U$-bootstrap is  the same that is required to update the occupation variable for the KCM.

In \cite{MMoT} two of the authors together with R. Morris addressed the problem of identifying the divergence of  time scales for two-dimensional KCM. 
The first goal of \cite{MMoT} was to identify the correct universality classes, which turn out to be different from those of  $\cU$-bootstrap percolation. Then, building on a strategy developed in \cite{MT} by two of the authors,
%club to analyse $r$-neighbour KCM (also known as Friedrickson Andersen spin facilitated models)
 universal upper bounds on the relaxation and mean infection time within each class were proven  and were conjectured  to be sharp up to logarithmic corrections \cite{MMoT}.
On the other hand, concerning lower bounds, so far the best general result  is \begin{equation}\label{eq:trivial_lower}\trel\geq q\mathbb E_{\mu}(\tau_0)=\Omega(T)\end{equation}  where $T$ denotes the median infection time for the $\mathcal U$-bootstrap process started with distribution $\mu$ (i.e. sites are initially infected independently with probability $q$), see \cite{MT}*{Lemma 4.3}.
However this lower bound is in general far from optimal.
Consider for example  the {\sl one-dimensional East model} \cite{JACKLE} (and  \cite{East-review} for a review) for which 
 a site can be updated iff its left neighbour is empty, namely $\mathcal U=\{\{-\vec e_1\}\}$.
As $q\downarrow 0$, it holds %$$\exp\bigg( \big( 1 + o(1) \big) \frac{\log(1/q)^2}{2\log2} \bigg)$$ (see~\cites{CFM,Aldous,CMRT}).
%\club{vogliamo mettere kla formula cosi' in evidenza o piuttosto nel testo? in quel caso possiuamo mettere "diverge as $q\downarrow 0$ as 
\begin{equation}\label{eq:East} \mathbb E^{\mbox{\tiny {East}}
}_{\mu}(\tau_0)=e^{(\Theta(\log q)^2)}\end{equation}
and the scaling holds for $\trel$, see~\cites{CFM,Aldous,CMRT} where the sharp value of the constant  has been determined.  This divergence is much faster than for the corresponding $\cU$-bootstrap model, for which it holds $T=\Theta(1/q)$.
To understand this difference it is necessary to recall a key combinatorial result  \cite{SE1},\cite{CDG}*{Fact 1}: in order to empty the origin the East process has to go through a configuration with $\lceil\log_2 (\ell+1)\rceil$ simultaneous empty sites in $(-\ell,0]$, where $-\ell$ is the position of the rightmost empty site on $(-\infty,0]$.
This 
logarithmic ``energy barrier'' (to employ the physics jargon)  and  the fact that  at equilibrium typically $\ell\sim 1/q$ yield a divergence of the time scale as $q^{\Theta(\log q)}=e^{(\Theta(\log q)^2)}$.
In turn, this peculiar scaling is the reason why the East model has been extensively studied by physicists (see  \cite{KGC} and references therein). Indeed, if we set $q:=e^{-\beta}$ with $\beta$ the inverse temperature, we get the so called  {\sl super-Arrhenius} divergence $e^{(\Theta(\beta^2))}$ 
%club (as opposed to the more standard Arrhenius scaling $\exp(O(1/T))$
which provides a very good fit of the experimental curves for fragile supercooled liquids near the glass transition \cite{Berthier-Biroli}. \\
%The classic experimental curve (Angel plot) for liquids near the glass transition shows that liquids can be classfied into two classes: strong li troppo lungo da spiegare dovremmo citare FA1f...
%\club{referzna on ci starebbe male, anche per quando lo citiamo a fine sezione supercritical}

In \cite{Robsurvey}, together with R. Morris, we conjectured that one of the universality classes of two-dimensional KCM, that we call {\sl supercritical rooted models}, features time scales diverging as  for the East model.
Our first main result  (Theorem \ref{thm:rooted})  is to establish a lower bound 
which allows  
together with the upper bound in \cite{MMoT}*{Theorem 1} to
prove this conjecture  \footnote{Actually, the conjecture in \cite{Robsurvey} states  that $\tau_0=e^{(\Theta(\log q)^2)}$ w.h.p. when $q\to 0$. As explained in Remark \ref{rem:w.h.p.}, we can also prove this stronger result.}, namely we prove
$$\mathbb E^{\tiny{\cU}}_{\mu}(\tau_0)=e^{(\Theta(\log q)^2)}\,\,\,\,\,\,\,\mbox{ $\forall\,\, \cU$ in the supercritical rooted class
}$$ and the same result for $\trel$.
As for the East model, this divergence is much faster than for the corresponding $\cU$-bootstrap process which scales as  $T=1/q^{\Theta(1)}$ \cite{BSU}.
A key input for our Theorem \ref{thm:rooted} is a combinatorial result  proved by one of the authors in \cite{Laure} (see also 
Lemma \ref{lem:laure} in this paper) which considerably generalises to a higher dimensional and non oriented setting the above recalled combinatorial result for East \footnote{The result in \cite{Laure} holds also in $d>2$ on a properly defined class, i.e. all models which are not  supercritical {\sl unrooted} (see \cite{Laure} for the precise definition). 
Our argument  immediately extends to this higher dimensional setting yielding the same lower bound as in Theorem \ref{thm:rooted} for $\trel$ and $\mathbb E_{\mu}(\tau_0)$.}.\\ %This result proves in particular that the universality classes of $\cU$-bootstrap models have to be refined

The $\cU$-bootstrap results identify  another universality class,
 the so called {\sl critical update families}, which display a much faster divergence. In particular, in \cite{BDMS} it was proven that for this class
it holds $T=e^{(\Theta(\log)^{c}/q^{\alpha})}$ with $\alpha$ a model dependent positive integer and $c=0$ or $c=2$.
In \cite{MMoT}, together with R.Morris, we analysed KCM with critical update families and we put forward the conjecture that both  $\trel$ and $\mathbb E_{\mu}(\tau_0)$ diverge as
$e^{(\Theta(\log)^{c'}/q^{\nu})}$  with $\nu$ in general different from the exponent $\alpha$ of the corresponding $\cU$-bootstrap process and we formulated (see \cite{MMoT}*{Conjecture 3}) a conjecture for the value of $\nu$ (which is again model dependent). In \cite{MMoT}*{Theorem 2} we established upper bounds for all critical models matching this conjecture. A matching lower bound exists only for those models for which the general lower bound \eqref{eq:trivial_lower} is sharp namely, in the language of \cite{MMoT},
 for the special case of $\beta$-unrooted models with $\beta=\alpha$.
Here we focus on the most studied update family which does not belong to this special case, the {\sl Duarte update family}, which consists of  all the
2-subsets of the North, South and West neighbours of the origin  \cite{Duarte}. Our second main result is a sharp lower bound on the infection and relaxation time for the Duarte KCM (Theorem \ref{thm:Duarte}) that, together with the upper bound in \cite{MMoT}, establishes the  scaling 
$$\bbE^{\mbox{\tiny{Duarte}}}_\mu(\t_0)=e^{\Theta\big((\log q)^4/q^2\big)}$$
as $q\downarrow 0$, and the same result holds for $\trel$.
The value $\nu=2$ for the exponent is in agreement with our conjecture  \cite{MMoT}*{Conjecture 3 (a)}, indeed in the language of \cite{MMoT} Duarte is a $1$-rooted model with $\alpha=1$, thus $\nu=2$. Notice  that we identify also the exact power in the logarithmic correction. Finally, notice that  the divergence is again much faster than for the corresponding $\cU$-bootstrap model.
Indeed, the median of the infection time for the $\cU$-bootstrap Duarte model diverges as $T=e^{(\Theta(\log q)^2/q)}$ when $q\downarrow 0$ \cite{Mountford} (see also  \cite{BCMS-Duarte} for the sharp value of the constant). \\

Both for Duarte and for supercritical rooted models, 
the sharper divergence of time scales for KCM is due to the fact that the infection time is not well approximated by the minimal number of updates needed to infect the origin (as it is for bootstrap percolation), but it is instead the result of a much more complex infection/healing mechanism. In particular, visiting regions of the configuration space with an anomalous amount of infection  is heavily penalised and requires a very long time to actually take place \footnote{Borrowing again from physics jargon we could say that ``crossing the energy barriers'' is heavily penalised.}.
%Providing an insight into the heuristics and/or the key steps of the proofs at this stage,  would be rather vague. We therefore defer these explanations to Section 4 and Section \ref{sec:reader}. 
%club o fare tutto qui? fra l'altro la spiegazione per supercritici a sto punto si puo togliere?
%A very high level explanation is  that
The basic underlying idea is that the dominant relaxation mechanism is an East like dynamics for large {\sl droplets} of empty sites. For supercritical rooted models these droplets have a finite (model dependent) size, hence an equilibrium density $q_{\mbox{\tiny{eff}}}= q^{\Theta(1)}$. For the Duarte model droplets have a size that diverges  as $\ell=\frac{| \log q|}{q}$ and thus an equilibrium density $q_{\mbox{\tiny{eff}}}=q^{\ell}=e^{-(\log q)^2/q}$.
Then a (very) rough understanding of our results is obtained by replacing $q$ with $q_{\mbox{\tiny{eff}}}$ in the result for the East model \eqref{eq:East}.
One of the key technical difficulties to translate this intuition into a  lower bound
is that the droplets cannot be identified with a rigid structure, at variance with the East model where the droplets are single empty sites. 
%The full difficulty is exemplified by Duarte KCM. The key steps of our strategy to overcome this obstacles are: (i) to  construct an algorithm  which allows to identify the droplets with the region that {\sl can} (without violating the constraints) create a certain rigid empty structure, essentially
% a vertical segment of height $\ell=1/q \log(1/q)$; (ii) notice that this empty segment  is {\sl tipically} able to empty a nearby segment of the same length; (iii) prove that
%%  (hence to facilitate the creation of another droplet to its right). In particular we show that 
%the dynamics projected on the droplets positions is an East dynamics; (iv)  build on the combinatorial results for East and on a variational lower bound for the infection time (see Section \ref{sec:general}) to establish the result.
%A more detailed explanation of the key steps is provided in Section \ref{sec:reader}.
%
\section{Models and notation}\label{sec:notations}
\subsection{Notation}
For the reader's convenience we gather here some of the
notation that we use throughout the paper. 
We will work on the probability space $(\O,\mu)$, where $\O=\{0,1\}^{\bbZ^2}$ and $\mu$ is the product
Bernoulli($p$) measure, and we will be interested in the asymptotic
regime $q\downarrow 0,$ where $q=1-p$. Given $\o\in \O$ and $\L\subset
\bbZ^2,$ we will often write $\o_\L$ or $\o\restriction_\L$ for the
collection $\{\o_x\}_{x\in \L}$ and we shall write $\o_\L\equiv
0$ to indicate that $\o_x=0\ \forall x\in \L$. In this case we shall also
say that $\L$ is empty or infected. Similarly for $\o_\L\equiv
1$ and in this case $\L$ will be said to be occupied or healthy. We
shall write $Y(\o)$ for the set $\{x\in \bbZ^2 \colon \o_x=0\}$ and
we shall say that $f \colon \O \mapsto \bbR$ is a \emph{local
  function} if it depends on finitely many variables $\{\o_x\}_{x\in\bbZ^2}$.
Given a
site $x\in \bbZ^2$ of the form $x=(a,b)$ with $a,b\in \bbZ,$ we shall
sometimes refer to $b$ as the height of $x$. 
We shall also refer to a set $I\subset
\bbZ^2$ of the form $I=\{x,x+\vec e_i,\dots,x+(n-1) \vec e_i\}, x\in \bbZ^2,$ as
a (horizontal or vertical) interval of length $n\in \bbN^*$.
 Finally, we will use the
standard notation $[n]$ for the set $\{1,\ldots,n\}.$

Throughout this paper we will often make use of standard asymptotic
notation. If $f$ and $g$ are positive real-valued functions of $q\in (0,1)$, then we will
write $f = O(g)$ if there exists a constant $C > 0$ such that $f(q)
\le C g(q)$ for every sufficiently small $q > 0$. We will also write
$f = \O(g)$ if $g = O(f)$ and $f= \Theta(g)$ if $f =
O(g)$ and $g = O(f)$. 
All constants, including those implied by the notation $O(\cdot)$,
$\O(\cdot)$ and $\Theta(\cdot)$, will be such w.r.t. the parameter $q$. 
%If $c_1$ and $c_2$ are constants, then $c_1 \gg c_2 \gg 1$ means that $c_2$ is sufficiently large, and $c_1$ is sufficiently large depending on $c_2$. Similarly, $1 \gg
%c_1 \gg c_2 > 0$ means that $c_1$ is sufficiently small, and $c_2$ is
%sufficiently small depending on $c_1$.

% In this section we define a %special 
% class of two-dimensional interacting particle systems known as \emph{kinetically constrained models}. As will be clear from what follows, KCM are intimately connected with bootstrap cellular automata.

\subsection{Models}\label{sec:models} Fix an {\sl update family} $\,\mathcal U=\{X_1,\dots,X_m\}$, that is,
a finite collection of finite subsets of $\mathbb Z^2\setminus \{
\mathbf{0} \}$. Then the KCM with update family $\cU$ is the Markov process on $\O$ associated to the Markov generator 
\begin{equation}
  \label{eq:generator}
(\cL f)(\o)= \sum_{x\in \bbZ^2}c_x(\o)\big( \mu_x(f) - f \big)(\o),
\end{equation}
where $f \colon \O \mapsto \bbR$ is a local function, $\mu_x(f)$
denotes the average of $f$ w.r.t.~the variable $\o_x$, and $c_x$ is
the indicator function of the event that there exists $X\in \cU$ such
that $X+x$ is infected \ie $\o_{X+x}\equiv 0.$ In the sequel we will
sometimes say that $\o$ satisfies the update rule at $x$ if
$c_x(\o)=1.$ 

Informally, this process can be described as follows. Each vertex
$x\in \bbZ^2$, with rate one and independently across $\bbZ^2$, is
resampled from $\big( \{0,1\},{\rm Ber}(p) \big)$ iff the update rule
at $x$ was satisfied by the current
configuration. In what follows, we will sometimes
call such resampling a \emph{legal update} or \emph{legal spin
  flip}. The general theory of interacting particle systems
(see~\cite{Liggett}) proves that $\cL$ becomes the generator
of a reversible Markov process $\{\o(t)\}_{t\ge 0}$ on $\O$, with reversible measure $\mu$. The corresponding
 Dirichlet form is
\[
\cD(f)= \sum_{x\in \bbZ^2}\mu\big(c_x \var_x(f)\big),
\] 
where $\var_x(f)$ denotes the variance of the local function $f$ w.r.t. the variable
$\o_x$ conditionally on $\{\o_y\}_{y\neq x}$. If $\nu$ is a
probability measure on $\O,$ the law of the process
with initial distribution $\nu$ will be denoted by $\bbP_\nu(\cdot)$
and the corresponding expectation by $\bbE_\nu(\cdot)$. If $\nu$ is
concentrated on a single configuration $\o$ we will simply write
$\bbP_\o(\cdot)$ and $\bbE_\o(\cdot)$.    

Given a KCM, and therefore an update family $\cU$,   the corresponding \emph{$\cU$-bootstrap process} on $\Z^2$ is defined as follows: given a set $Y \subset \Z^2$ of initially \emph{infected} sites, set $Y(0) = Y$, and define for each $t \geq 0$, 
\begin{equation}\label{eq:def:Uboot:At}
Y(t+1) = Y(t) \cup \big\{ x \in \Z^2 \,:\, X + x \subseteq Y(t) \text{ for some } X \in \cU \big\}.
\end{equation}
The set $Y(t)$ will represent the set of infected sites at time $t$
and we write $[Y] = \bigcup_{t \ge 0} Y(t)$ for the \emph{closure} of $Y$ under the $\cU$-bootstrap process. We will also call $T$ the median of the first infection time of the origin when the process is started with sites independently infected (healthy) with probability $q$ (respectively $p
=1-q$).

\section{A variational  lower bound for $\bbE_\mu(\t_0)$}
\label{sec:general}
%vogliamo cambiare titolo mettere A variational lower bound on $\bbE_\mu(\t_0)$ per mettere in evidenza che c'e' risultato?

As mentioned in the Introduction, our main goal
is to prove sharp lower bounds for the characteristic time scales of  supercritical rooted KCM and of the Duarte KCM. 
Let us start by defining precisely these time scales, namely the relaxation time $\trel$ (or inverse of the spectral gap) and the mean infection time $\bbE_\mu(\t_0)$.

\begin{definition}[Relaxation time, $\trel$]
\label{def:PC} Given an update family $\cU$ and $q\in[0,1]$,
we say that $C>0$ is a Poincar\'e constant for the corresponding KCM if, for all local functions $f$, we have
\begin{equation}
  \label{eq:gap}
\var_{\mu}(f) \leq C \, \cD(f).
\end{equation}
If there exists a finite Poincar\'e constant we then define
\[
\trel(q,\cU):=\inf\big\{ C > 0 \,:\, C \text{ is a Poincar\'e constant} \big\}.
\]
Otherwise we say that the relaxation time is infinite. We will drop the $(q,\cU)$ notation setting $\trel:=\trel(q,\cU)$ when confusion does not arise.
\end{definition}
A finite relaxation time implies that the reversible measure $\mu$ is mixing for the
semigroup $P_t=e^{t\cL}$ with exponentially decaying time
auto-correlations \cite{Liggett}.

\begin{definition}[Mean infection time, $\bbE_{\mu}(\tau_0)$]
Let  $A=\{\o\in \O:\ \o_0=0\}$. Then $$
\t_0 = \inf\big\{ t \ge 0 \,:\, \o(t)\in A \big\}.$$ 
Given an update family $\cU$ and $q\in[0,1]$, we let $\mathbb E_{\mu}^{q,\cU}(\tau_0)$ be the mean of the infection time of the origin under the corresponding stationary KCM  (i.e. when the initial configuration is distributed with Bernoulli$(1-q)$).  We will drop the $(q,\cU)$ notation setting $\mathbb E_{\mu}(\tau_0):=\mathbb E_{\mu}^{q,\cU}(\tau_0)$ when confusion does not arise.
\end{definition} 
In the physics literature the hitting time $\t_0$ is closely related
to the \emph{persistence time}, \ie the first time that there is a
legal update at the origin. All our lower bounds can be easily extended to the persistence time.

It is known that the following inequality holds (see \cite{MMoT}*{Section 2.2}): 
\begin{equation}
  \label{eq:mean-infection}
\bbE_\mu(\t_0) \leq \frac{\trel(q,\cU)}{q} \qquad \forall \; q\in (0,1).
\end{equation}
 Therefore we will focus on obtaining lower bounds on $\mathbb E_{\mu}(\tau_0)$ and then use \eqref{eq:mean-infection} to derive the results for $\trel$ (indeed the correction $q$ in the above inequality is largely subdominant w.r.t. the lower bounds we will obtain).
%The above inequality was extensively
%used in \cites{MT,MMoT} to establish universal upper bounds on
%$\bbE_\mu(\t_0)$ depending only on certain simple geometrical features of the
%update family $\cU$. 
To this aim we establish a variational lower bound on $\bbE_\mu(\t_0)$
(Lemma \ref{lem:basic:bound}), which will be our first tool.
Recall that $A=\{\o\in \O:\ \o_0=0\}$ and let $H_A$ be the Hilbert
space $\{f\in L^2(\O,\mu):
f\!\restriction_A=0\}$ with scalar product inherited from the
standard one in $L^2(\O,\mu)$. Let also $\cL_A$ be
the negative self-adjoint operator on $H_A,$ whose action on local
functions is given by
\[
\cL_A f(\o)= \id_{A^c}(\o)\cL f (\o).
\]
It turns out (see e.g. \cite{DaiPra}*{Section 3}) that, for any local function $f\in H_A$ and any $\o\in A^c$, 
\[
\bbE_\o \big(f(\o(t))\id_{\{\t_0>t\}}\big)= e^{t\cL_A}f(\o).
\]
In particular, by choosing $f=\id_{A^c}(\cdot),$ one gets
\[
\bbP_\mu(\t_0>t)= \int d\mu(\o)
\id_{A^c}(\o)e^{t\cL_A}\id_{A^c}(\o)=\langle
\id_{A^c},e^{t\cL_A}\id_{A^c}\rangle,
\]
where $\langle\cdot,\cdot\rangle$ denotes the scalar product on
$L^2(\O,\mu).$ Thus
\begin{equation}
  \label{eq:3}
\bbE_\mu(\t_0)= \int_{0}^\infty dt \ \langle \id_{A^c},e^{t\cL_A}\id_{A^c}\rangle\ge
\int_{0}^T dt \ \langle \id_{A^c},e^{t\cL_A}\id_{A^c}\rangle \quad \forall \ T>0.
\end{equation}
%\club c'e' bisogno di scrivere il lemma sotto forma variazionale o tutti capiscono (il perche' del titolo)?
\begin{lemma}
\label{lem:basic:bound}
Let $\phi\in H_A$ be a local function
such that $\mu(\phi^2)=1.$ Then   
\[
\bbE_\mu(\t_0)\ge T |\mu(\phi)|\Big(|\mu(\phi)| e^{-T\cD(\phi)}-
\big(T\cD(\phi)\big)^{1/2}\Big),\quad \forall \ T>0.
\]
\end{lemma}
\begin{proof}
Let $\phi\in H_A$ be as in the statement and write 
\[
\id_{A^c}= \a \phi + \psi, 
\]
where $\a=\langle \id_{A^c}, \phi\rangle=\mu(\phi)$ and
$\langle\phi,\psi\rangle=0$. Clearly $\langle \psi,\psi\rangle=\mu(A^c)-\a^2$. 
We claim that, for any $T>0$ and any $t\in [0,T],$
\begin{equation}
  \label{eq:1}
  \langle \id_{A^c},e^{t\cL_A} \id_{A^c}\rangle \ge \a^2e^{-T \cD(\phi)} -
  2|\a|\big(T\cD(\phi)\big)^{1/2},
\end{equation}
which, combined with \eqref{eq:3}, proves the lemma.
To prove the claim  we write 
\begin{align}
\label{eq:4}
\langle \id_{A^c},e^{t\cL_A} \id_{A^c}\rangle &\ge  \a^2 \langle
                                                \phi,e^{t\cL_A}\phi\rangle
                                                - 2 |\a|\, |\langle
                                                \psi,
                                                e^{t\cL_A}\phi\rangle
                                                |\nonumber \\
&=
\a^2 \langle \phi, e^{t\cL_A}
\phi\rangle - 2 |\a|\, |\langle \psi, (\bbI -e^{t\cL_A})\phi\rangle
  |\nonumber \\
&\ge \a^2 \langle \phi, e^{t\cL_A} \phi\rangle - 2 |\a|\, \langle\phi,\big(\bbI-e^{t\cL_A}\big)^2\phi\rangle^{1/2}.
 \end{align}
Above we discarded the positive term $\langle \psi, e^{t\cL_A}\psi\rangle$ in the first
line, we used
$\langle\phi,\psi\rangle=0$ in the
second line  and appealed to the Cauchy-Schwartz inequality together
with $\langle\psi,\psi\rangle\le 1$ in the
third line.
Let now $\pi(d\l)$ be the spectral measure of $-\cL_A$
associated to $\phi$ (see e.g. \cite{Reed-Simon}*{Chapter VII}). Since
$\mu(\phi^2)=1,$ $\pi(d\l)$ is a probability measure on $[0,+\infty)$.  
The functional calculus theorem, together with the Jensen inequality
and $(1-e^{-t\l})^2\le t\l,$ implies that for any $t\in [0,T]$
\begin{align*}
\text{r.h.s. \eqref{eq:4}}&=
\a^2 \int_0^\infty
  d\pi(\l) e^{-t\l} -2|\a|\, \Big(\int_0^\infty
  d\pi(\l)(1-e^{-t\l})^2\Big)^{1/2}\\
&\ge \a^2e^{-t\cD_A(\phi)}- 2|\a|\big(t\cD_A(\phi)\big)^{1/2}
\\
&\ge \a^2e^{-T\cD(\phi)}- 2|\a|\big(T\cD(\phi)\big)^{1/2},
  \end{align*}
where $\cD_A(\phi)=\langle \phi, -\cL_A\phi\rangle= \langle \phi, -\cL
\phi\rangle=\cD(\phi)$ because $\phi$ is a local function in $H_A.$ The claim is proved.
\end{proof}
The main strategy to take advantage of Lemma
\ref{lem:basic:bound} for $q$ very small is to look for a family of local functions
$\{\phi_q\}$ in $H_A$, normalised in such a way that
  $\mu(\phi_q^2)=1$, determining a sharp lower bound when 
   inserted in the inequality of Lemma \ref{lem:basic:bound} with a proper choice of $T$.
  More precisely we will use the following easy corollary of Lemma \ref{lem:basic:bound}:
  \begin{corollary}[Proxy functions]
  \label{cor:basic:bound}
  If there exists a family of local functions
$\{\phi_q\}$ in $H_A$ with $\mu(\phi_q^2)=1$ and
\begin{equation} 
  \label{eq:5}
\lim_{q\to 0}\cD(\phi_q)=0 \quad \text{and}\quad 
\lim_{q\to 0}\mu(\phi_q)^4/\cD(\phi_q)=+\infty.  
\end{equation}
then it holds \begin{equation}
  \label{eq:2}
\bbE_\mu(\t_0)=\O\Big(\mu(\phi_q)^4/\cD(\phi_q)\Big).
\end{equation}
\end{corollary}
\begin{proof} The result follows immediately using Lemma \ref{lem:basic:bound} and 
 choosing
$T\equiv T(q)=|\mu(\phi_q)|^2/(16\cD(\phi_q))$.\end{proof}
Any function $\phi=\phi_q$ with the above properties will be
called a \emph{test} or \emph{proxy} function and, in the rest of the paper, we will focus on constructing an efficient test
function for the so called \emph{supercritical rooted KCM} and for the \emph{Duarte KCM}.

\section{Supercritical rooted KCM}
%We begin by recalling the key geometric classification of two
%dimensional update families $\cU$ introduced by Bollob\'as, Smith and Uzzell~\cite{BSU}. 
% 
%\begin{definition}\label{def:Uboot}
%Given an update family $\cU = \{ X_1,\ldots,X_m \},$ the \emph{$\cU$-bootstrap process} on $\Z^2$ is defined as follows: given a set $Y \subset \Z^2$ of initially \emph{infected} sites, set $Y(0) = Y$, and define for each $t \geq 0$, 
%\begin{equation}\label{eq:def:Uboot:At}
%Y(t+1) = Y(t) \cup \big\{ x \in \Z^2 \,:\, X + x \subseteq Y(t) \text{ for some } X \in \cU \big\}.
%\end{equation}
%The set $Y(t)$ will represent the set of infected sites at time $t$
%and we write $[Y]_\cU = \bigcup_{t \ge 0} Y(t)$ for the \emph{closure} of $Y$ under the $\cU$-bootstrap process.
%\end{definition}
%
%Thus, a vertex $v$ becomes infected at time $t + 1$ if at time $t$ its
%constraint $c_v$ (see \eqref{eq:generator}) is equal to one, \ie if the translate
%by $v$ of one of the sets in $\cU$ is already entirely infected at time~$t$. Infected vertices remain infected forever. 
%
%One of the key
%insights of Bollob\'as, Smith and Uzzell~\cite{BSU} was that in  two dimensions the typical global behaviour of the
%$\cU$-bootstrap process acting on random initial sets is determined by the
%action of the process on discrete half-planes.
%For each unit vector $u \in S^1$, let 
%$\H_u := \{x \in \Z^2 : \< x,u \> < 0 \}$
%denote the discrete half-plane whose boundary is perpendicular to $u$. 
%
In order to define the class of {\sl supercritical rooted} update families we should begin by recalling the key geometrical notion of \emph{ stable directions} introduced in \cite{BSU}. Given a
  unit vector $u \in S^1$, let 
$\H_u := \{x \in \Z^2 : \< x,u \> < 0 \}$
denote the discrete half-plane whose boundary is perpendicular to $u$. 
Then, for a given update family $\cU$, the set of {stable directions} is
$$\stab = \stab(\U) = \big\{ u \in S^1 \,:\, [\H_u]= \H_u \big\}.$$
The update family $\cU$ is \emph{supercritical} if there exists an open semicircle in $S^1$ that is disjoint from $\cS$. In \cite{BSU} it was proven that for each supercritical update family the median of the infection time of the $\cU$-bootstrap processes diverges as $1/q^{\Theta(1)}$.
%\begin{itemize}
%\item \emph{supercritical} if there exists an open semicircle in $S^1$ that is disjoint from $\cS$, 
%The update family $\cU$ is:
%\begin{itemize}
%\item \emph{supercritical} if there exists an open semicircle in $S^1$ that is disjoint from $\cS$, \vspace{0.1cm}
%\item \emph{critical} if there exists a semicircle in $S^1$ that has finite intersection with $\cS$, and if every open semicircle in $S^1$ has non-empty intersection with $\cS$, \vspace{0.1cm}
%\item \emph{subcritical} if every semicircle in $S^1$ has infinite intersection with $\cS$. %\vspace{0.1cm}
%\end{itemize}
%\end{definition}
In \cite{Robsurvey}, the author R. Morris together with two of us, conjectured that
not all supercritical update families give rise to the same scaling for KCM
and that the supercritical class
should  be refined into two subclasses to capture the KCM scaling as follows.
\begin{definition}\label{def:rooted}
A supercritical two-dimensional update family $\cU$ is said to be
\emph{supercritical rooted} if there exist two non-opposite stable directions in
$S^1$. Otherwise it is called \emph{supercritical unrooted}.
\end{definition} 
An example of supercritical rooted family is the two dimensional East model, with update family   $\cU=\{\{-\vec e_1\},\{-\vec e_2\}\}$ \footnote{We stress that the supercritical rooted class contains also update families which do not share the special "orientation" property  of the East model, namely the fact that all $X_i$ belong to an half plane. For example, it is easy to verify that the non oriented update family $\,\cU=\{\{-\vec e_1\},\{-\vec e_2\}, \{(\vec e_1,\vec e_2)\}\}$ has exactly two stable directions, $-\vec e_1$ and $-\vec e_2$ and, according to our Definition \ref{def:rooted}, it is  supercritical rooted.}. 
%club should we also give a non oriented example?
In \cite{MMoT} it was proved that $\bbE_\mu(\t_0)$ and $\trel$ diverge as an
inverse power of $q$ as $q\to 0$ in the supercritical {\sl unrooted}  case,
while in the {\sl rooted} case it satisfies (see \cite{MMoT}*{Theorem 1 (b)})
\[
\trel\le e^{O((\log q)^2)}
\]
and, thanks to \eqref{eq:mean-infection}, the same bound holds for $\bbE_\mu(\t_0)$.
Here we prove a matching lower bound in the
rooted case. 
\begin{theorem} 
\label{thm:rooted}Let $\cU$ be a two dimensional supercritical rooted update family. Then  
  \[
 \bbE_\mu(\t_0)\ge e^{\O((\log q)^2)} \quad \text{as }q\to 0 .
\]
\end{theorem}
Thus we prove 
\begin{corollary}\label{cor:rooted}
 Let $\cU$ be a two dimensional supercritical rooted update family. Then  
  \[
\trel(q,\cU) = e^{\Theta((\log q)^2)} \quad \text{as }q\to 0 .
\]
and the same result holds for $ \bbE_\mu(\t_0)$.
\end{corollary}
\begin{proof}[Proof of the corollary]
The lower bound follows at once from \eqref{eq:mean-infection} and
Theorem \ref{thm:rooted}. The upper bound was proved in
\cite{MMoT}*{Theorem 1 (b)}.   
\end{proof}

%\club lascisamo  la dipendenza da U e q?
In order to prove Theorem  \ref{thm:rooted} we will  use the variational lower bound of Section \ref{sec:general} and more precisely
look for a proxy function $\phi\equiv \phi_q$ satisfying the key hypothesis of Corollary \ref{cor:basic:bound}.
 We
first need to introduce the notion of a legal path in $\O$. 
\begin{definition}[Legal path] 
\label{def:legal}Fix an update family $\cU$, then a legal path $\g$ in $\O$ is a finite sequence $\g=\big(\o^{(0)},\dots,\o^{(n)}\big)$ such
  that, for each $i\in [n],$ the configurations $\o^{(i-1)},\o^{(i)}$
  differ by a legal (with respect to the choice $\cU$) spin flip at some vertex $v\equiv
  v(\o^{(i-1)},\o^{(i)})$. A generic ordered (along $\g$) pair of consecutive
  configurations in $\g$
  will be called an \emph{edge}. Given a set $\hat \O\subset \O$ and a
  configuration $\o,$ we say that $\o$ is  a legal path
  connecting  $\hat \O$ to $\o$ if there exists a legal path
  $\g=\big(\o^{(0)},\dots,\o^{(n)}\big)$ such that $\o^{(0)}\in \hat
  \O$ and $\o^{(n)}=\o$. 
\end{definition}
Let $\cU$ be a supercritical rooted update family and, %  and let $\cC\subset S^1$ be an open semicircle free of stable
% directions. Since $\cU$ is rooted there exist (at
% least) two distinct non opposite stable directions $u,v\in
% S^1\setminus \cC$. 
for $n\ge 1$ and $\kappa\in \bbN^*$, let
$\L_n:=\L_n(\kappa)\subset \mathbb Z^2$ be the square centred at the origin,
of cardinality $(\kappa
n2^n+1)^2.$ Let also
\begin{equation}
  \label{eq:6}
  \begin{split}
\cA_n=&\{\o\in \O\colon
(\o_{\L_n},\tilde\o_{\L^c_n}\equiv 0)
\text{\it \ can be
reached from $\big(\hat\o_{\L_n}\equiv 1,
  \hat\o_{\L^c_n}\equiv 0\big)$ by a
legal}\\&
\text{\it path $\g$ such that any $\o'\in \g$ has at most $n-1$ empty vertices in
$\L_n$}\}.
  \end{split}
\end{equation}
Recall that $A=\{\o\in \O\colon \o_0=0\}$. In \cite{Laure} one of the authors established the following key combinatorial result concerning the structure of the
set $\cA_n$ : % In the sequel, given $\o\neq \o',$ we shall write
% $\o\leftrightarrow \o'$ if there exists $x\in \bbZ^2$ such that
% $c_x(\o)=1$ and $\o'$ is obtained from $\o$ by flipping $\o_x$.
\begin{lemma}[\cite{Laure}*{Theorem 1}]\label{lem:laure} There exists $\kappa_0=\kappa_0(\cU)>0$ such that, for
  any $\kappa\ge \kappa_0$ and any $n\in \bbN^*,$ 
\[
\cA_n\cap A=\emptyset.
\] 
% where 
% \[
% \partial \cA_n:=\{\o\notin \cA_n:\ \exists\,
% \o'\in \cA_n \text{ such that } \o\leftrightarrow \o'\}.
% \]
\end{lemma}
%club mettere anche qui footnote su higher dimension?
Lemma \ref{lem:laure} implies that the
KCM process started from any configuration with no infection inside the region $\L_n,$ in
order to infect the origin  has to leave
the set $\cA_n$ by going through its boundary set $\partial\cA_n$ (see the proof below for a precise definition of this  set).
In turn, the latter is a subset of 
\[
\{\o\in \O:\ \exists  \text{
  at least $n-1$ infected vertices in $\L_n$}\}.
\]
We will therefore chose a scale $n$ such that $2^n \simeq 1/q^{\epsilon}$, namely w.h.p. w.r.t. the reversible
measure $\mu$ there are initially no infected vertices inside $\L_n$. Thus,  
starting from the (likely) event of  no infection inside the region $\L_n$, in
order to infect the origin the process has to
go through $\partial \cA_n$ which has an anomalous amount, $\Theta(\log q)$, of empty sites.
This mechanism, which in the physics jargon would correspond to "crossing  an energy
barrier" which grows  logarithmically in $q$, is at the root of the scaling $e^{\Theta(\log q)^2}$. 
% Standard
%``activation energy'' 
%type of heuristics with $e^{-\b}:=q$ suggests that the time for the above process to
%occur is at least (at the dominant exponential growth) 
%\[
%\exp\Big(\b \times \O\big(\text{height of the
%  barrier}\big)\Big)= \exp\Big(\b\times
%\O\big(\log(1/q)\big)\Big)=\exp\Big(\O\big((\log q)^2\big)\Big).
%\] 
%\club{perche' non togliamo $\epsilon$ e lo fissiamo fin dall'inizion $<1/2$ e anche per duarte non bastava fissarlo $<1/4$?
Let us proceed to a proof of this result, namely to the proof of Theorem \ref{thm:rooted}.
\begin{proof}[Proof of Theorem \ref{thm:rooted}]
%Using the lemma we can finally define our proxy function $\phi$. 
Fix $\epsilon<1/2$ and choose $n:=n(\epsilon,q)=\lfloor \epsilon\log_2(1/q)\rfloor$. Then let $$\phi(\cdot):=\phi_q(\cdot)=\id_{\cA_{\epsilon,q}}(\cdot)/\mu(\cA_{\epsilon,q})^{1/2}$$
where $\cA_{\epsilon,q}:=\cA_{n(\epsilon,q)}$ with $\cA_n$ defined in \eqref{eq:6} and the constant $\kappa$ that enters in this definition chosen larger than the value $\kappa_0$
of Lemma \ref{lem:laure}.
Then Lemma \ref{lem:laure} implies immediately that $\phi\in H_A$.
Moreover, using $\epsilon<1/2$ we get 
\[
\mu(\phi)=\mu(\cA_{\epsilon,q})^{1/2}\ge (1-q)^{|\L_n|/2}=1-o(1),
\]
because any configuration identically equal to one in $\L_n$ belongs
to $\cA_{\epsilon,q}$ and $2^{2n}=O(1/q^{2\epsilon})$. 
Finally, if 
\[
\partial \cA_{\epsilon,q}:=\{\o\in \cA_{\epsilon,q}\colon \exists \ x\in \L_n \text{ with }
c_x(\o)=1\text{ and } \o^x\notin \cA_{\epsilon,q}\},
\] 
one easily
checks (see e.g. \cite{CFM3}*{Section 3.5}) that 
\begin{gather*}
\cD(\phi) \le |\L_n|\mu\big(\partial \cA_{\epsilon,q}\big)/\mu(\cA_{\epsilon,q})\le  |\L_n|\mu\big(\exists\  n-1 \text{
  zeros  in $\L_n$}\big)/\mu(\cA_{\epsilon,q})\\
\le O(|\L_n|^n) q^{n-1}=e^{-\O((\log q)^2)},
\end{gather*}
Thus $\phi$ satisfies all the hypotheses of Corollary \ref{cor:basic:bound} and the result follows.
\end{proof}

\begin{remark} \label{rem:w.h.p.} 
In
\cite{Robsurvey}*{Conjecture 2.7} it was conjectured that  $\tau_0=e^{\Theta((\log q)^2)}$ w.h.p. as $q\to 0$ holds. Actually, we can also prove this stronger  result.  One bound immediately follows using Markov inequality and our result for the mean, Corollary \ref{cor:rooted}. The other bound follows
by using the fact that (i) the set $\cA_{\epsilon,q}$  has $\mu$-probability $1-o(1)$ (see the above proof of Theorem \ref{thm:rooted})
and (ii) the probability of infecting the origin before $e^{\Theta((\log q)^2)}$ starting in $\cA_{\epsilon,q}$ goes to zero as $q\downarrow 0$. The latter result is easily obtained by a union bound on times which yields that the probability to leave $ \cA_{\epsilon,q}$ before $e^{\Theta((\log q)^2)}$ (and therefore to infect the origin, thanks to Lemma \ref{lem:laure}), goes to zero.
\end{remark}

%\club se vogliamo lasciare euristica direi che serve refernza a :"standardf activation time.."
%\club{se vogliamo risparmiare spazio possiamo togliere questa euristica (c'e' un po' nella intro) e casomai mettere solo "Lemma \ref{lem:laure} can in fact be interpreted as follows. The KCM process starting at $t=0$ from any configuration with no infection inside the region $\L_n,$ in
%order to infect the origin at the center of $\L_n$ has to leave
%the set $\cA_n$ by going through the boundary set $\partial\cA_n$.
%e magari dire analogo a east..?
\section{The Duarte KCM}
In this section we analyse the mean infection time for the
Duarte KCM. For this model the update family $\cU$ consists of the
$2$-subsets of the North, South and West neighbours of the
%\club dobbiamo mettere i vicini esplicitamente?
origin \cite{Duarte}. The infection time for the Duarte bootstrap process is known to scale   as $e^{\Theta((\log q)^2/q)}$ \cite{Mountford} (see also  \cite{BCMS-Duarte} for the sharp value of the constant).
Concerning the Duarte KCM,
  in \cite{MMoT}*{Theorem 2} it was proved
that 
\[
 \trel(q,\cU)\le e^{O\big((\log q)^4/q^2\big)} \quad \text{as }q\to 0.
\]  
and, thanks to \eqref{eq:mean-infection}, the same result holds for $\bbE_\mu(\t_0)$.
Here we establish a matching lower bound.
\begin{theorem}
\label{thm:Duarte}
Consider the Duarte KCM. Then
\[
\bbE_\mu(\t_0)\ge e^{\O\big((\log q)^4/q^2\big)} \quad \text{as }q\to 0.
\]  
\end{theorem}
Using \eqref{eq:mean-infection}, Theorem \ref{thm:Duarte} and
\cite{MMoT}*{Theorem 2} we get immediately the following
corollary.
\begin{corollary} For the Duarte KCM it holds
  \[
\trel(q,\cU)= e^{\Theta\big((\log q)^4/q^2\big)} \quad \text{as }q\to 0 .
\]
and the same result for $\bbE_\mu(\t_0)$.
\end{corollary}
Our result provides the first example of critical $\alpha$-rooted KCM for which the conjecture for the divergence of time scales that we put forward in \cite{MMoT}*{Conjecture 3 (a)} together with R. Morris
can be proven. Indeed, as explained in \cite{MMoT}, the Duarte model is a $1$-rooted model and the exponent $2$ that we obtain is in agreement with \cite{MMoT}*{Conjecture 3 (a)}.
In order to prove Theorem \ref{thm:Duarte} 
we will  start  by the variational lower bound of Section \ref{sec:general}, as for the supercritical rooted class.
However, defining the analog of
the set $\cA_n$  together with the test
function $\phi$ satisfying the hypotheses of Corollary \ref{cor:basic:bound} is much more
involved and it requires a subtle algorithmic construction.
Before explaining our construction it is useful to make some
simple observations on how infection propagates in the Duarte
bootstrap process. 
\subsection{Preliminary tools : the Duarte bootstrap process}
\label{sec:tools}

\begin{figure}[h]
\centering
\includegraphics[scale=0.3]{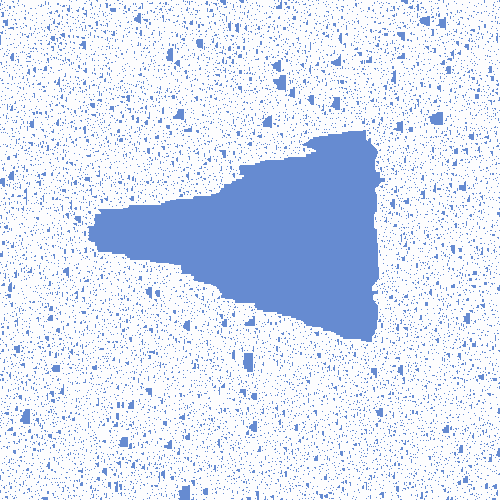}
\caption{A growing droplet under the Duarte bootstrap process (courtesy
  of P. Smith).}
\end{figure}

%\club Non sarebbe beglio (per future works..) al [osto di DB^{\tau}_{\Lambda} usare Duarte BP^{\tau}_{\Lambda}, for short BP^{\tau}_{\Lambda}
Let $\vec e_1,\vec e_2$ denote the basis vectors in $\bbR^2.$ Given $\L\subset \bbZ^2$ we write $\partial
\L:=\partial_\parallel\L\cup \partial_\perp\L,$ where
\begin{align*}
\partial_\parallel\L&=\{y\in \L^c \colon y+\vec e_1\in \L\},\\ 
\partial_\perp\L &=\{y\in \L^c \colon \{y+\vec e_2,y-\vec e_2\}\cap
                   \L\neq \emptyset\}.
\end{align*}
A configuration $\t\in \{0,1\}^{\partial \L}$ will be referred to as a
\emph{boundary condition} and we shall write it as
$\t=(\t_{\parallel}, \t_{\perp} )$, where
$\t_{\parallel}:= \t\restriction_{\partial_\parallel\L} $ and similarly for $\t_{\perp}$.
\begin{definition}
\label{def:fin-vol BP}
Given a boundary condition $\t$ and 
$Y\subseteq \L$, let 
\begin{equation*}
Y^\t(t+1) = Y^\t(t) \cup \big\{ x \in \L\,:\, X + x \subseteq Y^\t(t)
\text{ for some } X \in \cU \big\}\quad t\ge 0,
\end{equation*}
where $Y^\t(0)=Y\cup \{x\in \partial\L\colon \t_x=0\}$.
We call the process $Y^\t(t), t\in \bbN,$ the Duarte
bootstrap process in $\L$ with $\t$ boundary condition (for shortness the $DB_\L^\t$-process), and we shall write $[Y]_{\L}^{\t}$ for
$(\bigcup_{t \ge 0}Y^\t(t))\cap \L$. Recall also (see Section \ref{sec:models}) that $[Y]$ is the analogous quantity for the bootstrap process evolving on $\bbZ^2$.
\end{definition}
\begin{remark}
Notice that for the $DB_\L^\t$-process the boundary condition $\t$ does
not change in time. 
\end{remark}
\noindent
{\bf Notation warning.} If $\t\equiv 0$ or $\t\equiv 1$ we shall simply replace
it by a $0$ or a $1$ in our notation. If instead $\t$ is such that
$\t_\parallel\equiv 1$ and $\t_\perp\equiv 0$ then it will be replaced
by a $1,0$ in the notation.   
\begin{lemma}[Screening property]
\label{lem:screening}
Consider a sequence of sites $S:=\{(i,b_i)\}_{i=1}^n$ in $\bbZ^2$ with
$b_{i+1}\le b_i$ for all $i\in
[n-1],$ and let
\[
S_+=\{(i,j)\in \bbZ^2\colon i\in[n], j>b_i\},\quad
S_-=\{(i,j)\in \bbZ^2\colon i\in[n], j<b_i\}.
\] 
Let $Y,Y'$ be two arbitrary subsets of $\bbZ^2$ such that $Y\supseteq
S$ and  $Y\cap S^c_+=Y'\cap S^c_+.$
Then $[Y]\cap S_-=[Y']\cap S_-$. Similarly if we assume that
$b_{i+1}\ge b_i$ for all $i\in
[n-1]$ and we exchange the role of $S_+$ and $S_-$.
\end{lemma}
\begin{proof}
We refer to Figure \ref{fig:3bis} for a visualisation of the geometric
setting. Let $Y,Y'$ be as in the statement and observe that $Y(s)$
and $Y'(s)$ coincide in $\{v\in \bbZ^2\colon v=(a,b),\ a\le 0\}$ for
all $s\in \bbN^*$.
Let $t\in \bbN^*$ be the first time at which there exists $y\in S_-$ such
that either $y\in Y'(t)$ and $y\notin
Y(t)$ or viceversa. W.l.o.g we assume the first case. By construction there exists $z\in \{y\pm \vec e_2,y-\vec
e_1\}$ such that $z\in Y'(t-1)$ and $z\notin Y(t-1)$. Clearly $z$
cannot be of the form $z=(0,b)$ and therefore 
$z \in S_-\cup S$ because $y\in S_-$. Because of the definition of $t$,
$z\notin S_-$ and $z\notin S$ because $S\subseteq Y(s)$ and
$S\subseteq Y'(s)$ for all $s\in \bbN^*.$ 
\begin{figure}[ht]
  \centering
\begin{tikzpicture}[>=latex,scale=0.3]
[x=0.5cm, y=0.5cm]
\begin{scope}
\draw [help lines] (0,0) grid (9,12);
\foreach \i in {2,...,5}{%
\pgfmathparse{round(10-2*(\i^2/(3.2*\i+5))) }
       \let\theIntINeed\pgfmathresult
      \fill (\i-1,\theIntINeed) circle (3pt);
      \fill (\i+4,\theIntINeed-5) circle (3pt);
}
\draw [fill=gray, opacity=0.3] (1,8)--(2,8)--(2,7)--(4,7)--(4,6)--(5,6)--(5,3)--(7,3)--(7,2)--(9,2)--(9,0)--(1,0)--(1,8);
\fill (5,7) circle (3pt);
\node at (2.5,4.5) {$S_-$};
\node at (6.5,9.5) {$S_+$};
\node at (-0.5,9) {$b_1$};
\node at (10,3) {$b_n$};
\node at (1,-1) {$1$};
\node at (9,-1) {$n$};
\node at (4.5,11) {$\tiny x$};
\fill (5,11) circle (2pt);
\fill (0,0) circle (1.5pt);
\draw [fill=gray, opacity=0.3] (1,10)--(3,10)--(3,9)--(5,9)--(5,8)--(6,8)--(6,5)--(8,5)--(8,4)--(9,4)--(9,12)--(1,12)--(1,10);
\end{scope}
\end{tikzpicture}
\caption{The set $S$ (black dots) and the sets $S_\pm$ (shaded
  regions). If the two initial sets $Y,Y'$ of infection contain $S$
  and differ at
  exactly the vertex $x$, it is clear that the initial discrepancy cannot influence the final infection in $S_-$. }
\label{fig:3bis}
\end{figure}
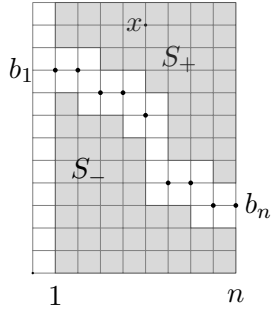
\end{proof}
\begin{lemma}[Monotonicity]
\label{lem:monotonicity}
 Let $\L\subseteq \L'$ be subsets of $\bbZ^2.$ 
\begin{enumerate}[(A)]
  \item Let
 $\t,\t'\in \{0,1\}^{\partial \L}$. If $\t_x\le \t'_x$ for all $x\in \partial\L$
then 
\[
[Y]_{\L}^{\t'}\subseteq [Y]_{\L}^{\t},\quad \forall \ Y\subseteq \L.
\]
\item For all $Y'\subseteq \L'$
\[
[Y']_{\L'}^{0}\cap \L \subseteq [Y'\cap \L]_{\L}^0
\quad\text{and}\quad [Y']_{\L'}^1\cap \L  \supseteq [Y'\cap
\L]^1_{\L}.
\]
\item Suppose that $\L$ and $\L'$ are such that $\partial_\perp
  \L\subseteq \partial_\perp
  \L'$. Then for all $Y'\subseteq \L'$
\[
[Y'\cap \L]_{\L}^{1,0}\subseteq [Y']_{\L'}^{1,0}\cap \L.
\] 
  \end{enumerate}
\end{lemma}
\begin{proof}\ 
  \begin{enumerate}[(A)]
\item It follows immediately from the fact that the
$DB_\L^\t$-process runs with more initial infection than the
$DB_\L^{\t'}$-process. 
\item To prove the first inclusion 
let $Z= (Y'\cap \L)\cup (\L'\setminus \L)$. Clearly $[Y']_{\L'}^0
\subseteq [Z]_{\L'}^0$ because $Y'\subseteq Z$.
It is now sufficient
to observe that, by definition, 
\[
[Z]_{\L'}^0\cap \L= [Y'\cap \L]_{\L}^0.
\]
Similarly
one proceeds for the second inclusion with $Z=Y'\cap \L$.
\item Clearly $[Y'\cap \L]_{\L'}^{1,0}
\subseteq [Y']_{\L'}^{1,0}.$ 
We claim that 
\[
[Y'\cap \L]_{\L'}^{1,0}\cap \L \supseteq [Y'\cap \L]_{\L}^{1,0}.
\]
That follows immediately from the assumption that
$\partial_\perp\L'\supseteq \partial_\perp\L$ and the fact that the
vertices of $\partial_\parallel\L\cap \L'$ (if any) are constrained to be healthy
for all times under  the $DB_\L^{1,0}$-process while they are
unconstrained for the $DB_{\L'}^{1,0}$-process.
\end{enumerate}
\end{proof}
\begin{lemma}[Propagation of infection]
\label{lem:propagation}
Let $I$ be a vertical interval, \ie $I=\{a,a+\vec e_2,\dots,a+n\vec
e_2\}, a\in \bbZ^2$, and let $v=x+\vec e_1$ for some $x\in
I$. Suppose that $I\cup \{v\}\subseteq [Y]$ where $Y$ is the initial set of
infection. Then $I+\vec e_1\subseteq  [Y]$. In particular, if $[Y]$ contains $[n]\times \{1\}$ and
$\{1\}\times [m]$ then
$[n]\times [m]\subseteq [Y]$. 
\end{lemma}
As a corollary of the
above simple property, let $x,y\in \bbZ^2$ and suppose that there
exists a \emph{Duarte path} $\G$ between $x$
and $y,$ \ie $\Gamma:=(x^{(1)},\dots,
x^{(n)})\subseteq \bbZ^2$ with $x^{(1)}=x,\ x^{(n)}=y$ and
$x^{(i+1)}-x^{(i)}\in \{\vec e_1, \pm
  \vec e_2\}\ \forall i\in [n-1].$  Let also $I_{\G}$ be the horizontal
  interval starting at $x$ and reaching the vertical line through $y$ (see Figure \ref{fig:3}).
\begin{corollary}
\label{cor:duarte-path}  
Suppose that $\G\subseteq [Y]$. Then $I_{\G}\subseteq
  [Y].$
\end{corollary}

\begin{figure}[ht]
  \centering
\begin{tikzpicture}[>=latex,scale=0.6]
[x=0.5cm, y=0.5cm]
\begin{scope}
\draw [help lines] (0,-2) grid [step=0.5] (10,2);
\draw [ultra thick] (0,0)--(2,0)--(2,2)--(3,2)--(3,-1)--(6,-1)--(6,0)--(7,0)--(7,2)--(8,2)--(8,-1.5)--(10,-1.5);
\draw [ultra thick, dotted] (0,0)--(10,0);
\draw [opacity=0.3,pattern=crosshatch] (2,0) rectangle (3,2);
\draw [opacity=0.3,pattern=crosshatch] (3,-1) rectangle (8,2);
\draw [opacity=0.3,pattern=crosshatch] (8,-1.5) rectangle (10,2);
 \node at (-0.3,0) {$x$};
 \node at (4.8,0.5) {$I_{\G}$};
 \node at (4.8,-1.5) {$\G$};
\node at (10.3,-1.5) {$y$};
\end{scope}
\end{tikzpicture}
\caption{A Duarte path $\G$ (thick polygonal line) and the corresponding horizontal
  interval $I_{\G}$ (dotted line). Clearly, $\G\subseteq [Y]$ implies
  that $[Y]$ contains the shaded region. In
  particular $I_{\G}\subseteq [Y]$.}
\label{fig:3}
\end{figure}

%\subsection{Heuristics of Duarte-KCM and sketch of the proof of Theorem \ref{thm:Duarte}}
%\label{sec:reader}
\subsection{Algorithmic construction of the test function and proof of Theorem \ref{thm:Duarte}}
\label{sec:Algo}  
Fix $\epsilon$ a small positive constant that will be chosen later on and
let 
%\begin{equation}\label{def:N}
%N= \big\lfloor e^{\e (\log q)^2/q}\big\rfloor  \quad {\mbox{ and }} \quad \ell = \Big\lfloor \frac{1}{\eps q}\log(1/q) \Big\rfloor,\quad
%\end{equation}
\begin{equation}\label{def:ell}
 \quad \ell = \Big\lfloor \frac{1}{\eps q}\log(1/q) \Big\rfloor.
\end{equation}
Suppose that a vertical interval $I$ of length $\ell$  is completely infected. Notice that, with $\mu$-probability going to $1$ as $q\downarrow 0$, there is an infected site on  the vertical interval sitting on the right, $I+\vec e_1$. Therefore,  thanks to Lemma \ref{lem:propagation}, with high probability the infection can propagate to infect $I+\vec e_1$.  Notice that instead the infection on $I$ does not help infecting the interval on its left, $I-\vec e_1$.  At this point, recalling the explanation given in the Introduction, one might think that the droplets that undergo an East like dynamics \footnote{namely a dynamics in which droplets appear/disappear only if there is a droplet on their left, as it occurs for the single empty sites in the East one-dimensional model.} are the {\sl empty vertical intervals of length at least $\ell$}.
However this is far from true, since these empty intervals might also appear (or disappear) without being facilitated by the presence of an empty interval on their left. For example, if there is an empty interval of length $\ell -1$ and the site just above has the constraint satisfied, a single legal move may turn it into an empty interval of height $\ell$.  We have therefore to find a more flexible definition of the droplets respecting three key properties: (i) East like dynamics ; (ii) disjoint occurrence under the equilibrium measure $\mu$ and (iii) the density of droplets should scale as $q_{\mbox{\tiny{eff}}}=q^{\ell}$ \footnote{Indeed, since the density of droplets will play the role of the density of empty sites for East, it is natural to expect that the lower bound obtained using the droplets will be of the form  \eqref{eq:East} with $q_{\mbox{\tiny{eff}}}$ replacing $q$. This in turn yields the result of Theorem \ref{thm:Duarte} if  $q_{\mbox{\tiny{eff}}}=q^{\ell}$.}. Our solution to the problem is the construction of an algorithm that sequentially searches for properly defined droplets on a finite volume, $V$, containing the origin.
We let
\begin{equation}\label{def:N}N= \big\lfloor e^{\e (\log q)^2/q}\big\rfloor  \quad {\mbox{ and }} \quad V:=V_N=\cup_{i=1\,}^{N} \mathcal{C}_i, 
\end{equation}
where
$$\cC_i= 
 \{(i,j)\in \bbZ^2 \colon |j| < N^2 -(i-1)N \}-N\vec e_1.
$$
as in Figure \ref{fig:5}.
In the sequel we shall write $\bar V$ for set $V\cup \partial_\perp V$
and we shall refer
to $\bar \cC_i:=\cC_i\cup \partial_\perp\cC_i$ as the $i^{th}$-column of $\bar V$.
By construction the origin coincides with the midpoint of
the last column (see Figure \ref{fig:5}). 
The core of our
algorithmic construction (see Definition \ref{def:algo}) consists in associating to each $\o \in \O$ an
element $\Phi(\o)\in \{\downarrow,\uparrow\}^N$ via an iterative 
procedure based on the   $DB_{\L}^{\tau}$-process. These arrow variables are those that
satisfy the three key properties announced above, with $\Phi(\o)_i=\uparrow$ corresponding to the occurrence of a droplet in column $i$, and we will use them to construct an efficient test function.
\begin{figure}[ht]
  \centering
\begin{tikzpicture}[>=latex,scale=0.3]
[x=0.5cm, y=0.5cm]

\begin{scope}
%\draw [help lines] (-10,-10) grid [step=1] (0,10);
\end{scope}
\begin{scope}
\foreach \i in {1,2,3,4,5}{%
\draw ({-10+(\i-1)*1.3},{-10+2*(\i-1)+0.7}) rectangle ({-10+(\i-1)*1.3+1},{10-2*(\i-1)-1.2});
\pgfmathsetmacro{\z}{2*(20-4*(\i-1)-1)}
\fill ({-10+(\i-1)*1.3+0.5}, {-10+2*(\i-1) +\z/2}) circle (5pt);
\fill ({-10+(\i-1)*1.3+0.5}, {-10+2*(\i-1) +1/2}) circle (5pt);

\foreach \j in {1,...,\z}{%
\fill [opacity=0.8] ({-10+(\i-1)*1.3+0.5}, {-10+2*(\i-1) +\j/2}) circle (3pt);
}%
}%

\draw [thick,dotted] (-11,0)--(-16,0);   
\draw[decorate,decoration={brace,mirror}] (-4.8,0.9)--(-4.8,2.7);
\node at (-3.5,2) {$N$};
\node at (-4,-2.8) {$\cC_N$};
\end{scope}
% \begin{scope}[shift={(7,0)},scale=0.8]
% \draw [thick] (1,9.5)--(3,9.5)--(3,1.5)--(5,1.5)--(5,-6.5);
% \draw [thick,dashed] (5,-6.5)--(7,-6.5);
% \draw [thick,dashed] (1,10)--(1,12);
% \foreach \j in {1,...,18}{%
% \draw [fill=lightgray] ({2}, {-9+\j}) circle (10pt);
% }%
% \foreach \j in {1,...,10}{%
% \draw [fill=lightgray] ({4}, {-9+\j}) circle (10pt);
% }%
% \draw [fill] (2,4) circle (10pt);
% \node at (2,-10) {$\cC_i$};
% \node at (5,-10) {$\cC_{i+1}$};
% \node at (1,4) {$v$};
% % \draw [thick,dotted] (2,-9)--(2,-11);
% % \draw [thick,dotted] (4,-9)--(4,-11);
%   \end{scope}
\end{tikzpicture}
\caption{A sketchy drawing of the last few columns of the set $V$. The black
  dots represents sites belonging to $\partial_\perp V$. % Suppose that
%   $Y\subseteq V$ contains the vertex $v$ (r.h.s). Then, for any
%   boundary condition $\t$, the set 
% $[Y]_{\cU,V}^\t$ does not depend on whether other sites of $\cC_i$
% above $v$
% belong to $Y$.
}
\label{fig:5}
\end{figure}
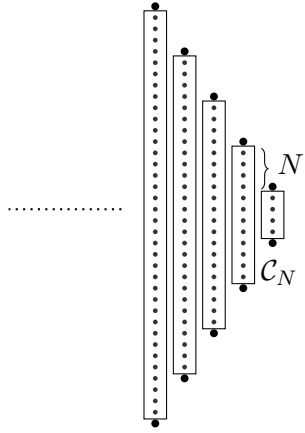
% Because of the definition of the Duarte update family $\cU$, when
% considering the $DB_V^\t$-process, the only values $\{\t_x\}_{x\in V^c}$
% that matter are those corresponding to $x\in \partial_\parallel
% V\cup \partial_\perp V$. In the sequel we will always restrict
% ourselves to boundary condition belonging to the set $\O_{\partial V}^1:=\{\t\in \{0,1\}^{\partial V}\colon \t
% \restriction_{\partial_\parallel V}\equiv 1\}.$ 
\begin{definition}
\label{def:UV}
Given a boundary condition $\t$ and $\o\in \O,$ we shall say that
$I\subseteq V$
is $(\o,\t)$-infectable if $I\subseteq\big [Y(\o)\cap
V\big]_{V}^{\t},$ where we recall that
$Y(\o)$ is the set of empty vertices of $\o$.
\end{definition}
Before defining the algorithm  leading to the construction of an
effective test function for the Duarte KCM process,
it is useful to notice two simple properties of the $DB_V^\t$-
process.  
\begin{enumerate}[(i)]
\item Let $I\subseteq \cup_{i=1}^k \cC_i, k\le N.$ Then the property of being
  $(\o,\t)$-infectable for $I$ depends only on the infection of the
  pair $(\o,\t)$ in $\cup_{i=1}^k \bar \cC_i$ and on $\t_\parallel$.  
\item If $\bar\cC_i$ is healthy at time $t=0$ (including the
  contribution of $\t$ at its top and bottom boundary sites), then it
  will remain healthy at any later time. 
\end{enumerate}

% Starting from a pair $\psi^{(0)}(\o):=
% (\o^{(0)}, \t^{(0)}),$ with $\o^{(0)}=\o$
\begin{definition}[The algorithm]
\label{def:algo}
Given $\o\in \O$ and $\t\in \{0,1\}^{\partial
  V}$ such that $\t_\perp \equiv 0$ and
$\t_\parallel\equiv 1$, the algorithm outputs recursively a sequence $\psi^{(k)}:=(\o^{(k)},
\t^{(k)}),\ k\in \{0,\dots,N\},$ where $\o^{(k)}\in \O$ and
$\t^{(k)}\in \{0,1\}^{\partial V}$ is such
that $\t^{(k)}_\parallel\equiv 1$.  The pair $\psi^{(0)}$ coincides
with $(\o,\t)$ and $\psi^{(k)}$ is obtained
from $\psi^{(k-1)}$ by healing suitably chosen infected
vertices. The iterative step goes as follows. Fix $\ell\in [N]$  and assume that
$\psi^{(j)}$ has been defined for all
$j=0,\dots, k-1, k\in [N].$ Then: 
\begin{enumerate}[(i)]
\item if $\bar \cC_k$ contains an interval $I$ of length
    at least $\ell$ which is 
    $\psi^{(k-1)}$-infectable, we let $\xi_k:=\xi_k(\o)\le k$ be
  the largest integer such that, by removing all the
  empty vertices of the pair $\psi^{(k-1)}$ contained in
  $\cup_{i=1}^{\xi_k-1}\bar \cC_i$, the above property still holds. We
  then set both $\o^{(k)}$ and $\t^{(k)}$ identically equal to one (\ie with no
  infection) on $\bar \cC_{\xi_k},\dots, \bar \cC_k$ and equal to
  $\o^{(k-1)}$ and $\t^{(k-1)}$ elsewhere; 
%   \begin{align}
% \label{eq:varphi}
%   \psi^{(k)}_x:=  (\o^{(k)}_x, \t^{(k)}_x)=
%     \begin{cases}
% (1,1) &\text{if $x\in \cup_{j=\xi_k}^k\cC_j$}\\
% (\o^{(k-1)}_x, \t^{(k-1)}_x)& \text{otherwise.}      
%     \end{cases}
%   \end{align}
\item if not we set $\psi^{(k)}=\psi^{(k-1)}.$
 \end{enumerate}
 \end{definition}
 \begin{remark}
   Clearly the above construction depends on the initial $\o$ and we
   shall sometimes write $\psi^{(k)}(\o)$ to outline this dependence. 
 \end{remark}
 \begin{definition}[Droplets and their range]
\label{def:droplets} Given $k$ such that $\psi^{(k)}(\o)\neq \psi^{(k-1)}(\o),$ we define the
\emph{droplet} $D_k(\o)$ and the \emph{range} $r_k(\o)$ of the
$k^{th}$-column in $\o$ as the set
$\cup_{i=\xi_k}^k \bar \cC_i$ and the integer $k-\xi_k(\o)$
respectively. If instead $\psi^{(k)}(\o)=\psi^{(k-1)}(\o),$ we let $D_k(\o)=
\emptyset$ and $r_k(\o)=0$.
 \end{definition}
Observe that, by construction,
\begin{equation}
  \label{eq:10}
\psi^{(j)}(\o)\restriction_{\bar V\setminus
  \cup_{i=1}^{j}D_i(\o)}=\psi^{(0)}(\o)\restriction_{\bar V\setminus
  \cup_{i=1}^{j}D_i(\o)}.
\end{equation}
\begin{definition}[The mapping $\Phi$] Having defined the sequence
  $\{\psi^{(k)}\}_{k=1}^N,$ we set 
  \begin{equation*}
\Phi(\o)_k = 
\begin{cases}
\uparrow &\text{ if $\psi^{(k)}(\o)\neq \psi^{(k-1)}(\o),$}\\
\downarrow &\text{ otherwise,}  
\end{cases}
  \end{equation*}
and $N_{\uparrow}(\o)=\#\{i\in [N]\colon \Phi(\o)_i=\uparrow\}$.
\end{definition}
\begin{remark}
\label{rem:algo}Suppose that $\o,\o'$ are such that they coincide over the first $i$
columns. Then
$\Phi(\o)_k=\Phi(\o')_k$ for all $k\in [i]$.  
\end{remark}
In the sequel two events will play an important role. 
%They are
%characterised by two new parameters $n_1,n_2$ which, once suitably
%chosen very large depending on $q,$ will make both events extremely unlikely. 
%The
%first one, $\cB_1(n_1),$ collects all the
%$\o'$s whose image $\Phi(\o)$ has more than $n_1$ up-arrows:  
%\begin{equation}
%  \label{eq:B1}
%\cB_1(n_1)=\{\o\in \O\colon N_{\uparrow}(\o)\ge n_1 \},\quad n_1\in [N].
%\end{equation}
%The event $\cB_2(n_2)$ collects instead all the $\o\in \O$ such that there
%exists $n_2$ consecutive $\downarrow$-columns which are traversed by
%an infectable Duarte path. More precisely, for $1\le i<j\le N,$ let
%$V_{i,j}=\cup_{k=i}^j \cC_k$ and let  
%\begin{equation}
%  \label{eq:B2}
%\cB_2(n_2)=\cup_{j-i\ge n_2-1}\big(
%\cap_{k=i}^j\{\o\in \O\colon \Phi(\o)_k=\downarrow\}\cap\cG_{i,j}\big), 
%\end{equation}
The
first one, $\cB_1(n),$ collects all the
$\o'$s whose image $\Phi(\o)$ has more than $n$ up-arrows, with $n\in [N]$:  
\begin{equation}
  \label{eq:B1}
\cB_1(n)=\{\o\in \O\colon N_{\uparrow}(\o)\ge n \}.
\end{equation}
The event $\cB_2(n)$, again with  $n\in [N]$, collects instead all the $\o\in \O$ such that there
exists $n$ consecutive $\downarrow$-columns which are traversed by
an infectable Duarte path. More precisely, for $1\le i<j\le N,$ let
\begin{equation}\label{def:Vij}V_{i,j}=\cup_{k=i}^j \cC_k\end{equation} and let  
\begin{equation}
  \label{eq:B2}
\cB_2(n)=\cup_{j-i\ge n-1}\big(
\cap_{k=i}^j\{\o\in \O\colon \Phi(\o)_k=\downarrow\}\cap\cG_{i,j}\big), 
\end{equation}
where
\begin{equation}
  \label{eq:B2bis}
\cG_{i,j}=\big\{\o\in \O\colon \exists \text{ a Duarte
  path }\G \text{ from }\cC_i \text{ to }\cC_j \text{ such that }
\G\subseteq [Y(\o)\cap V_{i,j}]^{1,0}_{V_{i,j}}\big\}
\end{equation}
% \begin{equation}
% \label{eq:B2}
% \cB_2:= \cB_2(n_2)=\{\o\in \O \colon \exists \, k\in [N] \text{
%   with } r_k(\o)\ge n_2\}.
% \end{equation}
We are now ready to define our test function.
\begin{definition}[The test function]
\label{def:tf}
Let $I_0=\{(0,k)\colon |k|\le \ell\}$ and 
\begin{align}
  \label{eq:8}
   n_1= \e (\log q)^2/2q,\quad n_2=1/q^6
\end{align}
where $\e$ is the same as in the definition of N \eqref{def:N}. Let also
$$\O_{\downarrow}=\{\o\in \O: \Phi(\o)=(\downarrow,\dots,\downarrow)\},$$
\[
\O_g=\O_{\downarrow}\cap\{\o\in \O\colon
  \o_{I_0}=1\},
\]
\begin{equation}
  \label{eq:7}
  \begin{split}
  \cA_{\epsilon,q}:=\cA_{N,\ell,n_1,n_2}
%=\{\o\in\O\colon \exists\eta\in \O_g  \text{ and a legal path } \gamma=(\omega^{(0)},\dots,\omega^{(n)}) \text{ s.t. } \\ \omega^{(0)}=\eta, \omega^{(n)}=\omega,
  =\{\o\in\O\colon \exists \text{ a legal path } \gamma \text{ connecting } \Omega_g \text{ to } \o \text{ s.t. }\\
   \g\cap  \cB_1(n_1-1)=\emptyset \text{ and }  \g\cap
    \cB_2(n_2-1)=\emptyset \}.   
  \end{split}
\end{equation}
where legal paths have been defined in Definition \ref{def:legal} and, for any  $\mathcal B\subset \Omega$, we set $\g\cap \mathcal B=\emptyset$ iff none of the configurations of the path $\gamma$ belongs to $\mathcal B$.
%\begin{equation}
%  \label{eq:7}
%  \begin{split}
%  \cA_{\epsilon,q}:=\cA_{N,\ell,n_1,n_2}
%=\{\o\in\O\colon \text{the configuration }(\o_{V^c}\equiv 0,\o_V) \text{ can
%    be reached from $\O_g$}\\
%\text{ by a legal path $\g$
%  such that $\g\cap
%    \cB_1(n_1-1)=\emptyset$ and } \g\cap
%    \cB_2(n_2-1)=\emptyset \}.   
%  \end{split}
%\end{equation}
Then we choose as test function
\[
\phi(\cdot):= \phi_q(\cdot)=\id_{\cA_{\epsilon,q}}(\cdot)/\mu(\cA_{\epsilon,q})^{1/2},
\]
%with 
%\begin{align}
%  \label{eq:8}
% N= \big\lfloor e^{\e (\log q)^2/q}\big\rfloor, \quad \ell = \Big\lfloor \frac{1}{\eps q}\log(1/q) \Big\rfloor,\quad
%  n_1= \e (\log q)^2/2q,\quad n_2=1/q^6,
%\end{align}
\end{definition}
The rest of the paper is devoted to prove that (i) $\phi$ satisfies all the hypotheses of Corollary \ref{cor:basic:bound}, namely $\phi\in H_A$ and the conditions \eqref{eq:5} are satisfied; (ii) $\phi$ is an efficient proxy function, namely the bound
\eqref{eq:2} prove the sharp lower bound of Theorem \ref{thm:Duarte}. 
More precisely we need to prove the following key propositions:
\begin{proposition}
\label{prop:P1}
There exists $\e_0>0$ such that, for all $\e\in (0,\e_0)$ there exists
$q_\epsilon$ small enough such that, for all $q\in (0,q_\epsilon)$,    
\[
\cA_{\e,q}\cap A=\emptyset.
\]
In particular, $\phi\in H_A$.
\end{proposition}
\begin{proposition}
\label{prop:P2}
There exists $\e_0>0$ such that, for all $\e\in (0,\e_0),$ 
\[
\mu(\phi)\ge q^{O(1)} \quad \text{and}\quad \cD(\phi)\le
e^{-\O(\log(q)^4/q^2)} \quad \text{as }q\to 0.
\] 
\end{proposition}
Once the above propositions are proven, the main result of this section easily follows
\begin{proof}[Proof of Theorem \ref{thm:Duarte}]
The result follows at once using Propositions
\ref{prop:P1} and \ref{prop:P2}, together with the general lower bound
on $\bbE_\mu(\t_0)$ given in \eqref{eq:2}.  
\end{proof}

Let us start with an easy result which will be used in the proof of both propositions
\begin{lemma}[Disjoint occurrence of the droplets]
\label{lem:D1}
For any $\o\in \O$ and any $k\neq j,$ $D_k(\o)\cap D_j(\o)=\emptyset$.
\end{lemma}
\begin{proof}
Let $k_1,\dots,k_\nu$ be the labels of the columns which are of type
$\uparrow$ in $\Phi(\o)$ (for all the other columns the droplets are
the empty set). Using property (ii) of the $DB_V^\t$-process, $D_{k_\nu}(\o)$ cannot contain a column which
    is healthy for the pair $\psi^{(k_\nu-1)}$ because any infection
    to the left of an healthy column cannot cross the healthy column
    itself. On the other hand, all the columns of the
droplets $D_{k_1},\dots,D_{k_{\nu-1}}$ are healthy for $\psi^{(k_\nu-1)}$. Thus $D_{k_\nu}\cap
    D_{k_{j}}=\emptyset$ for all $j\in [\nu-1].$ The same reasoning applies to all the
    other droplets.
 \end{proof}

\subsection{East-like motion of the arrows and proof of Proposition \ref{prop:P1}}
Let 
\[
A_\ell= \{\o\in \O\colon \o_{I^+_0}\equiv 0\}\cup \{\o\in
\O\colon \o_{I^-_0}\equiv 0\},
\]
where $I_0^\pm=\{(0,\pm 1),\dots,(0,\pm \ell)\}.$
Then it holds
\begin{lemma}\label{lem:pers}
If $\cA_{\e,q} \cap A\neq \emptyset$ 
then there exists $\o\in A_\ell$ and a legal path $\g$
connecting $\O_g$ to 
%$(\o_V,\o_{V^c}\equiv0)$ 
$\o$ such that 
%$\g$ never uses an update outside $V$ and 
$\g\cap
\cB_i(n_i)=\emptyset,\ i=1,2.$
\end{lemma}
\begin{proof}
Fix $\o\in \cA_{\e,q} \cap A$, recall Definition \ref{def:tf} and let
$\tilde\g$ be a legal path connecting
$\O_g$ to $\o$
%$(\o_V,\o_{V^c}\equiv 0)$ 
such that $\tilde \g\cap
\cB_1(n_1-1)=\emptyset$ and $\tilde \g\cap
\cB_2(n_2-1)=\emptyset.$ W.l.o.g., we can assume that $\tilde\g$ 
%never uses an update outside $V$ and that $\tilde\g$
ends as soon as it
enters $A$. It is easy to verify that $\tilde\g$ must be able to sequentially infect (and
possibly heal later on) the ordered vertices of either
$I_0^+$ starting from $(0,\ell)$ or those of $I_0^-$ starting from $(0,-\ell)$. For simplicity we assume that the first
option holds and we let $\g$ be the path obtained from $\tilde\g$ by
deleting all the transitions in which a vertex of $I_0^+$ is
healed. 

By construction, the final configuration of $\g$
belongs to $A_\ell$. Moreover,
$\g$ is a legal path because at each step
the infection in the last column of $V$ is larger than or equal to the
infection of the corresponding step of $\tilde\g$. Finally the restriction
to $\cC_1,\dots,\cC_{N-1}$ of any step of $\g$ coincides
with the same restriction of the appropriate step of $\tilde\g$. Using that $\tilde\g\cap
\cB_1(n_1-1)=\emptyset$ and $\tilde\g\cap
\cB_2(n_2-1)=\emptyset,$ we deduce that $\g\cap
\cB_1(n_1)=\emptyset$ and $\g\cap
\cB_2(n_2)=\emptyset$. 
\end{proof} 
The above Lemma says that, if there exists a configuration in $\Omega_g$ for which we can infect the origin performing a legal path never crossing either $\mathcal B_1(n_1-1)$ or $\cB_2(n_2-1)$, 
%which never goes neither through a configuration with more than $n_1-1$ up arrows, nor through a configuration with $n_2$ consecutive $\downarrow$ that are traversable by an infectable Duarte path, 
then necessarily there exists a legal path never crossing either $\mathcal B_1(n_1)$ or $\cB_2(n_2)$ and connecting a configuration $\omega$ with all columns being $\downarrow$ to a configuration $\omega$ with a $\uparrow$ in the $N$-th column.
 In order to conclude that $\mathcal A_{\epsilon,q}\cap A=\emptyset$ and thus prove our Proposition \ref{prop:P1}, we will now 
show that {\sl the existence of a legal path with the above
properties is impossible}. It is here that the East-like motion of the droplets emerges and plays a key role.
Recall the definitions  \eqref{def:N}, \eqref{eq:8} and let $m=4n_1 n_2$ and, for simplicity, let us suppose that
$m$ divides $N$. We partition $[N]$ into $M=N/m$ disjoint consecutive blocks
$\{B_i\}_{i=1}^M$ of equal cardinality and, with a slight abuse of
notation, we identify the columns $\cup_{k\in B_i}\cC_k$ with the
block $B_i$ itself. Given $\o\in \O$ we write
\[
\eta_i(\o):=\id_{\{\exists \,j\,\in B_i\colon \Phi(\o)_j=\uparrow\}},
\]
and we denote by $\eta(\o)$ the collection 
$\{\eta_i(\o)\}_{i=1}^M.$
\begin{claim}
\label{claim:20} Given a legal path $\g$ with the properties stated in Lemma \ref{lem:pers},  it is possible to
construct a path
$\varphi(\g):=(\eta^{(0)},\dots,\eta^{(k)})$ in the space $\{0,1\}^M$
with the following properties:
\begin{enumerate}[(1)]
\item $\eta^{(0)}_i=0$ for all $i\in [M]$ and $\eta^{(k)}_M=1$,
\item $\#\{i\in [M]\colon \eta_i=1\}\le n_1$ for all $\eta\in
  \varphi(\g)$, 
\item for any edge $(\eta,\eta')$ of $\varphi(\g),$ the
  configuration $\eta'$ differs from $\eta$ in exactly one
  coordinate. Moreover, if the discrepancy between $\eta$ and $\eta'$
  occurs at the $i^{th}$-coordinate and $i\neq 1,$ then $\eta_{i-1}=1$.
\end{enumerate}
\end{claim}
\begin{remark}
The path $\varphi(\g)$ for the coarse-grained variables
$\{\eta_i\}_{i=1}^M$ can be viewed as a legal path for the one dimensional East chain on $[M],$ with
facilitating vertices those for which $\eta_i=1$ (see
e.g. \cite{East-review}).  
\end{remark}
The proof of our Proposition \ref{prop:P1} then follows by using this connection with the East chain, our choices \eqref{def:N}, \eqref{eq:8} of the parameters $N,n_1,n_2$
and the combinatorial result for the East model \cite{SE1,CDG} that we explained in the Introduction. More precisely 
\begin{proof}[Proof of Proposition \ref{prop:P1}] In
\cite{CDG} it was proved that a path like $\varphi(\g)$ above exists
iff $n_1\ge \log_2(M +1).$ With our choice \eqref{eq:8} of the scaling as $q\to 0$ of
$n_1,n_2,N,$ the latter condition becomes
\[
n_1\ge \frac{1}{\log 2}(1+o(1))\e (\log q)^2/q,\quad \text{as } q\to 0,
\]
violating our choice $n_1=\e (\log q)^2/2q$. Thus $\varphi(\g)$
cannot exist as well as the path $\g$. 
\end{proof}

We are therefore left with proving Claim \ref{claim:20}.
To this aim we start by stating two preparatory results, Lemma \ref{lem:D0} and Lemma \ref{lem:D2}, which will be the key ingredients for the proof of  Claim \ref{claim:20}.

\begin{lemma}
\label{lem:D0}
For any $\o\in \cB_2^c(n_2)$ the maximum range of a droplet of $\o$ is $n_2-1$.  
\end{lemma}
\begin{proof}
Let $\omega \in \Omega$ such that there exists $j \in [N]$ with $r_j(\omega) \geq n_2$. 
   Denote $i = \xi_j(\omega)$.    By the definition of $\xi_j(\omega)=i$, $\bar{\mathcal{C}}_j$ contains an interval $I$ 
   of length at least $\ell$ which is $\psi^{(j-1)}$-infectable by the empty sites in 
   $\bigcup_{k=i}^j \bar{\mathcal{C}}_k$, but not by the empty sites in 
   $\bigcup_{k=i+1}^j \bar{\mathcal{C}}_k$. Definition \ref{def:UV} implies that 
   any $\psi^{(j-1)}$-infectable site is in $V$, hence $I \subseteq \mathcal{C}_j$. 
   Furthermore, for all $k \in \{i,\dots,j-1\}$, $\Phi(\omega)_k=\downarrow$
   (since thanks to Lemma \ref{lem:D1} the droplets are disjoint), so 
   by (\ref{eq:10}) $\psi^{(j-1)}$ and $\psi^{(0)}$ coincide on $\bigcup_{k=i}^j \bar{\mathcal{C}}_k$. 
   Therefore $I$ is $\psi^{(0)}$-infectable by the empty sites in 
   $\bigcup_{k=i}^j \bar{\mathcal{C}}_k$, but not by the empty sites in 
   $\bigcup_{k=i+1}^j \bar{\mathcal{C}}_k$. We deduce that $I \subseteq 
   [Y(\omega)\cap V_{i,j}]_{V_{i,j}}^{1,0}$, but $I \not\subseteq 
   [Y(\omega)\cap V_{i+1,j}]_{V_{i+1,j}}^{1,0}$, see \eqref{def:Vij} for the definition of $V_{i,j}$. Thus, there exists $z \in \mathcal{C}_j$ 
   such that $z \in [Y(\omega)\cap V_{i,j}]_{V_{i,j}}^{1,0} \setminus [Y(\omega)\cap V_{i+1,j}]_{V_{i+1,j}}^{1,0}$. Hence
   $z$ can not be initially empty for the Duarte bootstrap process in $V_{i,j}$, otherwise it would also be 
   empty for the process in $V_{i+1,j}$, hence the process in $V_{i,j}$ infects $z$ with an update rule, 
   so there exists $z' \in \{z-\vec{e_1},z\pm\vec{e_2}\}$ in 
   $[Y(\omega)\cap V_{i,j}]_{V_{i,j}}^{1,0} \setminus [Y(\omega)\cap V_{i+1,j}]_{V_{i+1,j}}^{1,0}$. 
   We can iterate, creating a Duarte path in $[Y(\omega)\cap V_{i,j}]_{V_{i,j}}^{1,0} \setminus 
   [Y(\omega)\cap V_{i+1,j}]_{V_{i+1,j}}^{1,0}$. There can be only a finite number of iterations because 
   there is a finite number of sites in $V_{i,j}$, so we will stop, and the site at which we stop 
   has to be initially empty for the process in $V_{i,j}$, but not for the process in 
   $V_{i+1,j}$, therefore it is in $\bar{\mathcal{C}}_i$. This implies the Duarte path can reach 
   $\mathcal{C}_i$. Consequently, there is a Duarte path in $[Y(\omega)\cap V_{i,j}]_{V_{i,j}}^{1,0} \setminus 
   [Y(\omega)\cap V_{i+1,j}]_{V_{i+1,j}}^{1,0}$ going from $\mathcal{C}_i$ to $\mathcal{C}_j$. 
   We deduce that there exists a Duarte path in $[Y(\omega)\cap V_{i,j-1}]_{V_{i,j-1}}^{1,0}$ 
   from $\mathcal{C}_i$ to $\mathcal{C}_{j-1}$, which is $\mathcal{G}_{i,j-1}$. 
   Since $(j-1)-i \geq n_2-1$, $\omega \in \mathcal{B}_2(n_2)$.

\end{proof}
The next lemma is the basic technical step connecting the evolution of
the coarse-grained variables $\{\Phi(\o)_i\}_{i=1}^N$ under the
Duarte KCM process to an East-like
process. Given $\o\in \O$ and $x\in V,$ let $\o^x$ denote the
configuration $\o$ flipped at $x.$ We say that $x$ is
$\psi^{(k)}(\o)$-unconstrained (or infectable in one step) if $\exists X\in \cU$ such that
$X+x$
is infected for the pair $(\o^{(k)},\t^{(k)})$. 
% Let also $\bar D_i(\o):=
% D_i(\o)\cup\cC_{(\xi_i(\o)-1)\vee 1}$ where, we recall, $\cC_{\xi_i(\o)-1}$ is
% the column adjacent (to the left) to $D_i(\o).$

 \begin{lemma}[East like motion of the arrows]
 \label{lem:D2}  
Fix $\o\in \O$ and let $x\in
\cC_j$. Then: 
\begin{enumerate}[(a)]
\item Suppose that $x$ is $\psi^{(0)}(\o)$-unconstrained. Then $\Phi(\o^x)\neq
  \Phi(\o)$ implies that $j>1$ and $\Phi(\o)_{j-1}=\uparrow;$      
\item For $i>j$ suppose that 
% either $j=1$ or $j>1$ and $x$ is
%   $\psi^{(0)}(\o)$-infectable and that
$\Phi(\o)_i=\uparrow, \Phi(\o^x)_i=\downarrow$ and that $D_i(\o)\not\ni x$. Then there exists
$k$ such that $\bar \cC_k\subseteq D_i(\o)\setminus \bar \cC_i$ and % either $\Phi(\o^x)_k=\uparrow$ or 
 $\Phi(\o^x)_k=\uparrow,\Phi(\o)_k=\downarrow$ .
\end{enumerate}
 \end{lemma}
 \begin{proof}\ 

(a) If $j=1$ then clearly
$\Phi(\o^x)=\Phi(\o)$ because $x$ is $\psi^{(0)}(\o)$-unconstrained.
Consider now the case $j\neq 1$ and assume that
$\Phi(\o)_{j-1}=\downarrow$. We want to prove that in this case
$\Phi(\o^x)=\Phi(\o)$ if $x$ is $\psi^{(0)}(\o)$-unconstrained. 

By construction, the restriction to the first $j-1$ columns of
$\psi^{(k)}(\o^x)$ and $\psi^{(k)}(\o)$ coincide for all $k\in [j-1]$
and, as a consequence, $\Phi(\o)_k=\Phi(\o^x)_k\, \forall k\in [j-1]$. Let $k_*(\o)=\min\{k\ge j\colon
  \Phi(\o)_k=\uparrow\}$ and similarly for $\o^x$. Using \eqref{eq:10} together with
  $\Phi(\o)_{j-1}=\downarrow$, for all $i=j-1,\dots, k_*(\o)-1$ the restriction of
  $\psi^{(i)}(\o)$ to the columns $\bar \cC_{j-1},\dots,\bar\cC_N$
coincides with the same
  restriction of the original pair $\psi^{(0)}(\o)$. In particular, the
  fact that $x$ is $\psi^{(0)}(\o)$-unconstrained implies that $x$ is
also  $\psi^{(k_*(\o)-1)}(\o)$-unconstrained. Analogously for the
configuration $\o^x$. Clearly $k_*(\o^x)\ge k_*(\o).$ If not, starting
from the infection of $\psi^{(j-1)}(\o)$ we can first
make a transition to $\psi^{(j-1)}(\o^x)$ by legally flipping $\o_x$ and from
there infect an interval of length at least $\ell$ of $\bar\cC_{k_*(\o^x)}$ to make it of type
$\uparrow,$ a contradiction with the definition of
$k_*(\o)$. By exchanging the role of $\o,\o^x$ we conclude that
$k_*(\o^x)=k_*(\o).$ Thus $\Phi(\o)_k=\Phi(\o^x)_k$
  for all $k=1\dots,k_*(\o)$ and, a fortiori, for all $k>k_*(\o)$. 

(b) % Suppose first that $j>1$ and $x$ is
%   $\psi^{(0)}(\o)$-infectable. Part (a) and the assumption $\Phi(\o)\neq \Phi(\o^x)$ imply
% that $\Phi(\o)_{j-1}=\uparrow.$ In particular, $\cC_{j-1}\notin
% D_i$.
By assumption the restriction of $\o,\o^x$ to $D_i(\o)$
coincide. % Let $m\equiv \xi_i(\o)-1$ be the label of the first column
% of $\bar D_i(\o)$. If
% $\Phi(\o)_m=\uparrow$ we are finished.
If $\Phi(\o^x)_k=\downarrow$ for all the columns in $D_i(\o),$ then  
% %instead $\Phi(\o)_m=\downarrow,$ then
% % necessarily $j\le m$ and therefore 
% there must exists $k\in (m,\dots,i-1)$ such that
% $\Phi(\o^x)_k=\uparrow$ (and $\Phi(\o)_k=\downarrow$), since otherwise
$\psi^{(i-1)}(\o)=\psi^{(i-1)}(\o^x)$ on the set $D_i(\o)$
implying that $\Phi(\o^x)_i=\Phi(\o)_i.$ Thus there exists a column
$\bar\cC_k\subseteq D_i(\o)\setminus \bar\cC_i$ such that $\Phi(\o^x)_k=\uparrow$
and (by the definition of $D_i(\o)$) $\Phi(\o)_k=\downarrow$.
% Consider now the case $j=1$. If $\bar D_i\not\ni \cC_1$ then we can
% repeat the previous argument. If $\bar D_i\ni \cC_1$    
\end{proof}

\begin{corollary}
\label{cor:D1}
Fix $\o\in \O$ and let $x\in
\cC_j.$ Let also 
$r^x_\infty=\max_i\max(r_i(\o),r_i(\o^x))$ and suppose that $\Phi(\o)_i=\uparrow,\Phi(\o^x)_i=\downarrow,$ with $i-j\ge m(r^x_\infty +1), m\in
\bbN^*$. Then
\[
\#\{k\in \{j,\dots,i\}\colon \Phi(\o)_k=\uparrow\}+\#\{k\in \{j,\dots,i\}\colon \Phi(\o^x)_k=\uparrow\}\ge m.
\]
\begin{proof}
By construction $D_i(\o)\not\ni x.$ % If there exists $k\in \bar D_i\setminus \cC_i$ such that
% $\Phi(\o)_k=\uparrow$ then we can shift $i$ to $k$ and start again. If
% instead $\Phi(\o)_k=\downarrow$ for all $k\in \bar D_i\setminus \cC_i,$
% then 
Lemma \ref{lem:D2} part (b) guarantees that there exists a column
$\bar \cC_k\subseteq D_i(\o)\setminus \bar \cC_i$
such that $\Phi(\o)_k=\downarrow$ and $\Phi(\o^x)_k=\uparrow.$ We can
then iterate by exchanging the role of $\o,\o^x$ and replacing $i$
with e.g. the largest of the labels $k$
above. In conclusion, every $r^x_\infty+1$ steps we are guaranteed to
find a discrepancy between $\Phi(\o)$ and $\Phi(\o^x)$ and the result follows.
\end{proof}
\end{corollary}

We are now ready to conclude the proof of Claim \ref{claim:20}.
\begin{proof}[Proof of Claim \ref{claim:20}]
To prove the claim, let $\g=(\o^{(0)},\dots, \o^{(n)})$ and let us
consider the sequence $\{\eta(\o^{(j)})\}_{j=0}^n$. The path
$\varphi(\g)=(\eta^{(0)},\dots,\eta^{(k)})$ is then defined
recursively by setting $\eta^{(0)}:=\eta(\o^{(0)})$ and
$\eta^{(j)}:=\eta(\o^{(i_j)}),$ where $i_j=\min\{i>i_{j-1}\colon
  \eta(\o^{(i)})\neq \eta^{(j-1)}\}$ with $i_0=0,$ and by stopping the procedure as soon as
  the set $\{\eta\in \{0,1\}^M\colon \eta_M=1\}$ is reached. In other words, we only keep
  the elements of the 
  sequence $\eta(\o^{(j)}),j=0,\dots,n,$ which change w.r.t. the
  previous element.
Properties (1)  of $\varphi(\g)$ follows immediately from the
fact that $\g$ starts in $\O_{\downarrow}$ and ends in $A_{\ell}$. Property (2) follows from the fact that
$\g\cap \cB_1(n_1)=\emptyset$. We now
verify the key property (3). 

Let $(\eta,\eta')$ be an edge of $\varphi(\g)$
and let $(\o,\o')$ be the edge of $\g$ such that
$\eta(\o)=\eta$ and $\eta(\o')=\eta'$. By construction $\Phi(\o)\neq
\Phi(\o')$. Let also $x\in \cC_a$ be such
that $\o'=\o^x$ and say that $a$ belongs to $j^{th}$-block. Clearly, $\eta_i=\eta_i'$ for all
$i<j$. Moreover, Corollary \ref{lem:D0} and Corollary \ref{cor:D1}
  imply that $\Phi(\o)_v=\Phi(\o')_v$ for all $v\in
  \cup_{i\ge j+2}B_i$ (if $j+2\le N$), since otherwise either $\o$ or $\o'$ would have at
  least $\lfloor m/2(r^x_\infty+1)\rfloor \ge \lfloor
  m/2n_2\rfloor =2n_1$ up-arrows, contradicting the assumption $\g\cap
  \cB_1(n_1)=\emptyset$. In particular, $\eta_i=\eta'_i$ for all $i\ge
  j+2$. To complete our analysis we distinguish between two cases.   
\begin{enumerate}[1)]
\item $a>1.$ In this case $x$ must be $\psi^{(0)}(\o)$-unconstrained and part (a) of Lemma \ref{lem:D2} together with
  $\Phi(\o)\neq \Phi(\o')$ implies that
  $\Phi(\o)_{a-1}=\Phi(\o^x)_{a-1}=\, \uparrow.$ If
  $a$ is not the beginning of the block $B_j$ then, by definition, $\eta_j=\eta'_j=1$. Thus $\eta,\eta'$ must differ exactly in
  the $(j+1)^{th}$-block and they are both equal
  to one in previous one as required. If $a$ is the beginning of the
  $j^{th}$-block, then necessarily $j>1$. Moreover
  $\Phi(\o)_{a-1}=\Phi(\o^x)_{a-1}=\uparrow$ implies that
  $\eta_{j-1}=\eta'_{j-1}=1$. By the same reasoning as before, using Corollary
  \ref{cor:D1} and  Lemma \ref{lem:D0} (recall that  $\omega\in \cB^c(n_2)$)  we get that $\Phi(\o)_v=\Phi(\o')_v$ for all $v\in
  \cup_{i>j}B_i.$ Thus $\eta_i=\eta'_i$ for all $i\neq j$ and $\eta_{j-1}=\eta'_{j-1}=1$ as required.
\item $a=1$. Again Corollary
  \ref{cor:D1} guarantees that $\Phi(\o)_i=\Phi(\o^x)_i$ for all $i\in
  \cup_{j=2}^NB_j$ so that $\eta_b =\eta'_b$ for all $b\ge 2$.  
\end{enumerate}
\end{proof}

\subsection{Density of droplets and proof of Proposition \ref{prop:P2}}

The core of the proof of Proposition \ref{prop:P2} consists in bounding from
above the probabilities of the events
$\cB_1,\cB_2$ defined in \eqref{eq:B1},\eqref{eq:B2}.  
The first key bound is Lemma \ref{lem:D4}, that says that the probability
that the $DB^{1,0}_V$-process restricted to an arbitrary number of
consecutive columns of $V$ is able to infect any given interval of the last column of
length $\ell$ is $e^{-\O((\log q)^2/q)}$. The second key ingredient is
Lemma \ref{lem:D5} that bounds from above the
probability of the event $\cB_2(n_2-1)$.
Before stating the lemmas we need some additional notation. 

Given $1\le i\le j\le N,$ 
let $\L =\cup_{k=i}^j 
\cL_k,$ where, for each $k=i,\dots,j,$ $\cL_k\supseteq \cC_k$ is a (finite)
interval of $\{(k-N,j)\colon j\in \bbZ\}.$
% Recall that $Y(\o)$ is the set of infected vertices
% of $\o$ and fix a boundary configuration $\t$ on $\partial \L$ such
% that $\t_\parallel \equiv 1$. 
Let also $I\subseteq \cC_j$ be an
interval of length $\ell$. The
basic event that we
will consider is  
\[
\cO^\t_\L(I)=\{\o\in \O\colon I\subseteq [Y(\o)\cap \L ]_{\L}^\t\},
\]
where we recall $Y(\o)$ is the set of infected vertices of $\o$.  
Notice that $\cO^\t_\L(I)$ is an increasing event (\ie its indicator function is
an increasing function) 
w.r.t. to the partial order: $\o\prec \o'$ iff $\o'_x\le \o_x\
\forall x.$ 
Our first main lemma reads as follows. 
\begin{lemma}[Density of up-arrows]
\label{lem:D4}  
Choose the basic scales $N,\ell,n_1,n_2$ as in
 \eqref{def:ell},\eqref{def:N} and \eqref{eq:8}. Then there exists $c>0$ such that, for any $\e>0$
  sufficiently small and any $1\le i\le j\le N$,
\[
\max_I\mu(\cO^{1,0}_{V_{i,j}}(I))\le e^{-c(\log q)^2/q},\quad \text{as }
q\to 0,
\]
where $V_{i,j}=\cup_{k=i}^j\cC_k$.
\end{lemma}
\begin{proof}[Proof of Lemma \ref{lem:D4}]
Fix $1\le i\le j\le N$ together with an interval $I\subset \cC_j$ of
length $\ell$ and let 
\[
\L_{1,j}=\cup_{i=1}^{j}\{(i,k)\colon |k|< N^2\}-N\vec e_1.
\] 
We first claim that 
\begin{align}
  \label{eq:12}
 \mu(\cO^{1,0}_{V_{i,j}}(I))\le \mu(\cO^{1,0}_{V_{1,j}}(I))\le O(1/q^2)\mu(\cO^1_{\L_{1,j} }(I)) \quad \text{as }
q\to 0.
\end{align}
The first inequality follows from (C) in Lemma \ref{lem:monotonicity}. To prove the second one, let $G=\cap_{k=1}^{j-1} G_k,$ where $G_k$
denotes the event that there is an empty site within the
first $\lfloor N/3\rfloor $ sites and within the last $\lfloor
N/3\rfloor $ sites of $\cC_k$.
Then, for any choice of the constant $\e$ appearing in \eqref{eq:8}, 
\begin{equation}
  \label{eq:21bis}
\mu(G^c)\le 2N(1-q)^{\frac N3 -1}=o(1) \quad \text{as } q\to 0.
\end{equation}
For any $\o\in G$ and any boundary condition $\t$ for $V_{1,j}$ such that $\t\equiv 0$ on
$\partial_\perp \cC_j$ and $\t_\parallel\equiv 1,$ the screening
property and translation invariance imply that 
$[\,Y(\o)\cap V_{1,j}\,]_{V_{1,j}}^\t \cap \cC_j$ does not depend on $\t.$ Hence,
\begin{equation}
  \label{eq:18}
\cO^{1,0}_{V_{1,j}}(I)\cap G=\cO^\t_{V_{1,j}}(I)\cap G.
\end{equation}
Choose $\t$ equal to one everywhere except for $\partial_\perp \cC_j$
where it is equal to zero. Using the FKG inequality and \eqref{eq:18},
\begin{align*}
\mu\big(\cO^{1,0}_{V_{1,j}}(I)\big)&\le
                                     \mu\big(\cO^{1,0}_{V_{1,j}}(I)\tc
                                     G\big) =\mu(\cO^{\t}_{V_{1,j}}(I)\tc G)\\
&\le (1+o(1)) \mu\big(\cO^{\t}_{V_{1,j}}(I)\big).
  \end{align*}
We now observe that, starting from $Y(\o),$  we can construct the set $[Y(\o)\cap V_{1,j}]_{V_{1,j}}^\t\cap
\cC_j$ as follows. We first output the set $[Y(\o)\cap V_{1,j-1}]_{V_{1,j-1}}^1$
and we let $\bar \t\in \{0,1\}^{\partial \cC_j}$ be such
that $\bar \t_\perp\equiv 0$ and
$\{x\in \partial_\parallel \cC_j\colon \bar\t_x=0\}= [Y(\o)\cap
V_{1,j-1}]_{V_{1,j-1}}^1\cap \partial_\parallel\cC_j.$ Then we
output the set $[Y(\o)\cap \cC_j]_{\cC_j}^{\bar\t}$ which clearly coincides with $[Y(\o)\cap V_{1,j}]_{V_{1,j}}^\t\cap
\cC_j$. 

Monotonicity and a moment of thought imply that if we repeat the above
construction with $V_{1,j-1},\cC_j$ replaced by
$\L_{1,j-1},\ \{(j-N,k)\colon |k|< N^2\}$ and 
$Y(\o)$ replaced by $Y(\o)\cup \partial_\perp\cC_j,$ then the final
infection in $\cC_j$ cannot decrease. Hence
\begin{align*}
\mu\big(\cO^{\t}_{V_{1,j}}(I)\big)\le \mu\big(\cO^1_{\L_{1,j} }(I)\tc
  \o_{\partial_\perp \cC_j}\equiv 0\big)\le \mu\big(\cO^1_{\L_{1,j} }(I)\big)/q^2,
  \end{align*}
and \eqref{eq:12} follows.

Let now $T(\cU)$ be the median of the infection time of the origin (or of any
other vertex of $\bbZ^2$  because of translation invariance) for the Duarte bootstrap
process in $\bbZ^2$ started from $Y(\o)$ where $\o$ has law $\mu,$ and write 
\begin{equation}
  \label{eq:13}
p(N,\ell):=\max_{j\le N}\max_I\mu(\cO^1_{\L_{1,j}}(I)),  
\end{equation}
where
$\max_I$ is taken over all intervals $I\subset \cC_j$ of length $\ell$. 
\begin{claim}
\label{claim:1}If $\e<1/4$ then, for all $q$ small enough, 
\begin{equation}
  \label{eq:16}
p(N,\ell)\ge e^{-\frac{1}{16q}\log(q)^2 },
\end{equation}
implies 
\begin{equation*}
  T(\cU)\le O(N^3) e^{\frac{1}{16q}\log(q)^2 }.
\end{equation*}
\end{claim}
Before proving the claim we conclude the proof of Lemma
\ref{lem:D4}. It follows from the main result of \cite{BCMS-Duarte}
together with a standard (and straightforward) argument that 
\[
T(\cU)\ge e^{(1-o(1))\log(q)^2/8q }\quad \text{as }q\to 0,
\]
implying that for all $q$ small enough 
\[
p(N,\ell)\le e^{-\frac{1}{16q}\log(q)^2 },
\]
% the above contradicting \eqref{eq:15}
if $\e<1/48.$ 
\end{proof}
\begin{proof}[Proof of the claim]
In the sequel it will help to refer to Figure \ref{fig:2} as a visual
guide for the various definitions. 
Fix $q$ arbitrarily small and let $j$ be such that there exists an interval
$I\subset \cC_j$ of length $\ell$ such that 
\begin{equation}
  \label{eq:pippo}
\mu(\cO^1_{\L_{1,j}}(I))\ge e^{-\frac{1}{16q}\log(q)^2 }.
\end{equation}
Using the symmetry w.r.t. the horizontal axis we can assume that $x_I$, the lowest site of $I,$ has
non positive height. Write $\L^{(i)}:=\L_{1,j}-ij\vec e_1$ and let
$\cM_t=\cup_{i=0}^t\L^{(i)}, $ where $t=10\lceil
\max(p(N,\ell)^{-1},8/q^4)\rceil.$ 
\begin{figure}[ht]
  \centering
\begin{tikzpicture}[>=latex,scale=0.8]
[x=0.5cm, y=0.5cm]
\begin{scope}
\draw [help lines, opacity=0.3] (-11,0) grid [step=0.2] (2,6);
\foreach \i in {0,1,2,3,4,5} {%
\draw [thick] (-{\i*2.2},0) rectangle (-{\i*2.2}+2,6); 
}%

\foreach \j in {5,6,7,8,9}{%
\fill (2,{2+\j*0.2}) circle (1.7pt);
\draw [fill=lightgray] (-9,{2+\j*0.2}) circle (2pt);
%\draw [fill=lightgray] (-7,{2+\j*0.2}) circle (2pt);
\draw [fill=lightgray] (-6.8,{2+\j*0.2}) circle (2pt);
}%

% sparse initial empty sites
\draw [fill=lightgray] (-9+0.2,3.2) circle (1pt);
\draw [fill=lightgray] (-9+0.4,3) circle (1pt);
\draw [fill=lightgray] (-9+0.6,3.4) circle (1pt);

\begin{scope}[shift={(0.6,0.2)}];
\draw [fill=lightgray] (-9+0.2,3.2) circle (1pt);
\draw [fill=lightgray] (-9+0.4,3) circle (1pt);
\draw [fill=lightgray] (-9+0.6,3.4) circle (1pt);
\end{scope}

\begin{scope}[shift={(1.2,0)}];
\draw [fill=lightgray] (-9+0.2,3.2) circle (1pt);
\draw [fill=lightgray] (-9+0.4,3) circle (1pt);
\draw [fill=lightgray] (-9+0.6,3.4) circle (1pt);
\draw [fill=lightgray] (-8.2,3.6) circle (1pt);
\end{scope}

\draw [thick,dotted,->](-8.2,3.8)--(-7.5,3.8);
\draw [thick,dotted,->](-8.2,3)--(-7.5,3);

\foreach \j in {5,6,7,8,9,10}{%
\draw [fill=lightgray] (-5.8,{2+\j*0.2}) circle (2pt);
}%

\draw [thick,dotted,->](-6.4,3.6)--(-5.9,3.6);
\draw [thick,dotted,->](-6.4,3.2)--(-5.9,3.2);

\foreach \j in {5,6,7,8,9,10,11}{%
\draw [fill=lightgray] (-5,{2+\j*0.2}) circle (2pt);
}%

\draw [thick,dotted,->](-5.6,4)--(-5.1,4);
\draw [thick,dotted,->](-5.6,3.6)--(-5.1,3.6);

\foreach \j in {5,6,7,8,9,10,11,12}{%
\draw [fill=lightgray] (-3.8,{2+\j*0.2}) circle (2pt);
}%

\draw [thick,dotted,->](-4.6,4)--(-4,4);
\draw [thick,dotted,->](-4.6,3.6)--(-4,3.6);

\foreach \j in {5,6,7,8,9,10,11,12,13}{%
\draw [fill=lightgray] (-3.2,{2+\j*0.2}) circle (2pt);
}%

\foreach \j in {5,6,7,8,9,10,11,12,13,14}{%
\draw [fill=lightgray] (-2.4,{2+\j*0.2}) circle (2pt);
}%

%frecce finali
\draw [thick,dotted,->](-2.2,4.8)--(1.8,4.8);
\draw [thick,dotted,->](-2.2,3)--(1.8,3);

% increasing stair
\fill (-7+0.2,4) circle (2.3pt);
\fill (-6+0.2,4.2) circle (2.3pt);
\fill (-5.2+0.2,4.4) circle (2.3pt);
\fill (-4+0.2,4.6) circle (2.3pt);
\fill (-3.4+0.2,4.8) circle (2.3pt);
%\fill (-1+0.2,5) circle (2.3pt);

% final interval
\foreach \j in {5,6,7,8,9,10,11,12,13,14}{%
\draw[opacity=0.5,fill=lightgray] (2,{2+\j*0.2}) circle (2.3pt);}

% sparse empty sites

\draw [fill=lightgray] (-2.2,3.2) circle (1pt);
\draw [fill=lightgray] (-2,3.6) circle (1pt);
\draw [fill=lightgray] (-1.8,4) circle (1pt);
\draw [fill=lightgray] (-1.6,3.8) circle (1pt);
\draw [fill=lightgray] (-1.4,3.2) circle (1pt);

\begin{scope}[shift={(0,0.4)}];
\draw [fill=lightgray] (-1.2,3.2) circle (1pt);
\draw [fill=lightgray] (-1,3.6) circle (1pt);
\draw [fill=lightgray] (-0.8,4) circle (1pt);
\draw [fill=lightgray] (-0.6,3.8) circle (1pt);
\draw [fill=lightgray] (-0.4,3.2) circle (1pt);
\end{scope}
\begin{scope}[shift={(1,-0.2)}];
\draw [fill=lightgray] (-1.2,3.2) circle (1pt);
\draw [fill=lightgray] (-1,3.6) circle (1pt);
\draw [fill=lightgray] (-0.8,4) circle (1pt);
\draw [fill=lightgray] (-0.6,3.8) circle (1pt);
\draw [fill=lightgray] (-0.4,3.2) circle (1pt);
\end{scope}

\begin{scope}[shift={(2,0.6)}];
\draw [fill=lightgray] (-1.2,3.2) circle (1pt);
\draw [fill=lightgray] (-1,3.6) circle (1pt);
\draw [fill=lightgray] (-0.8,4) circle (1pt);
\draw [fill=lightgray] (-0.6,3.8) circle (1pt);
\draw [fill=lightgray] (-0.4,3.2) circle (1pt);
\draw [fill=lightgray] (-0.2,2.6) circle (1pt);
\end{scope}

%left dots
\draw [ultra thick, dotted] (-11.2,3)--(-12,3);

%labels
\node at (-10,2) {$\L^{(\nu)}$};
\node at (1,2) {$\L_{1,j}$};
\node at (-1,2) {$\L^{(1)}$};
\node at (-5.4,2) {$\L^{(i)}$};
%\node at (3,3.4) {$x_I$};
%\draw [thick,->] (2.7,3.4)--(2.1,3.4);
\node at (2.3,3.4) {$\Big\}$};
\node at (2.6,3.4) {$I$};
\node at (3.2,3.9) {$\hat I$};
\node at (2.8,3.9) {$\Biggg\}$};
\end{scope}
\end{tikzpicture}
\caption{A subset of the collection of boxes $\L^{(i)}$ forming $\cM_t$. On the
  last column of $\L_{1,j}$ the two intervals $\hat I\supset I$. The
  little gray dots denote suitable sparse single infected sites, one
  for each relevant column, and they have been drawn only for the
  initial and final stage of the infection process. The large gray dots on the right boundary of $\L^{(\nu)}$ represent a shifted copy
  of $I$ which is infected by  the $DB_{\L^{(\nu)}}^{1}$-process. This infected interval propagates to the right
  until reaching the first site of the empty
  upward stair (black dots). At this stage the interval grows
  vertically by one unit. This process continues until the interval
  has become a shifted copy of the interval $\hat I$. The latter interval is able to continue moving to the
right until infecting the interval $\hat I$.}
\label{fig:2}
\end{figure}
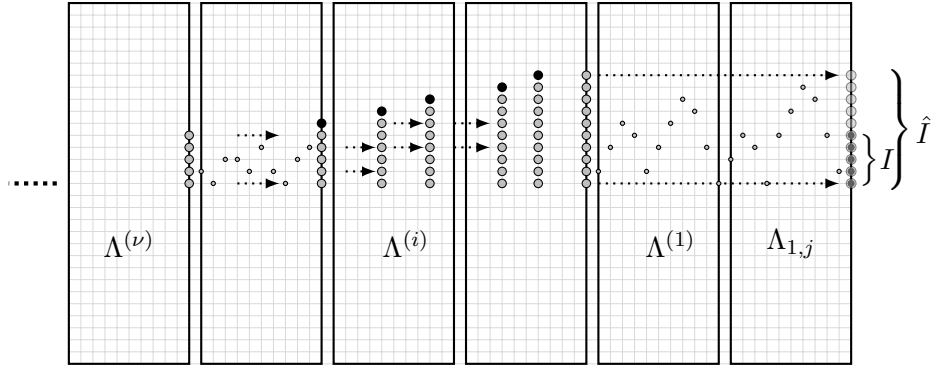
We shall
define two increasing events $\cG_1,\cG_2\subset \O,$ depending only on
$\o\restriction_{\cM_t},$ such that:
\begin{enumerate}[(a)]
\item if $\o\in \cG_1\cap \cG_2$ then the
  Duarte bootstrap process in $\bbZ^2$ is able to infect $x_I$ within time $2jt(2N^2-1)$. 
\item $\mu\big(\cG_k\big)>3/4,\ k=1,2$.
\end{enumerate}
Using the FKG inequality, 
$
\mu\big(\cG_1\cap \cG_2\big)\ge
\mu\big(\cG_1\big) \mu\big(\cG_2\big) > 1/2.
$
Hence
\[
T(\cU)\le 2jt(2N^2-1)\le 40 N^3 \Big(e^{\frac{1}{16q}\log(q)^2 }+1\Big).
\] 
In order to define $\cG_1,\cG_2,$
let $\hat I\supset I$ be the interval of $\cC_j$ of length $\lceil 1/q^3\rceil$
and whose lowest site is $x_I$. Then:
\begin{align*}
\cG_1&=\{\forall\, k\in [jt], \text{\it the
interval $\hat I-(k-1)\vec e_1$ contains an empty vertex}\};\\
\cG_2&=\{\exists\, k\in [jt]\colon \text{\it 
       the $DB_{\cM_t}^1$-process starting from $Y(\o)\cap \cM_t$ is able to infect $\hat I-k\vec e_1$}\}.
\end{align*}
We now verify properties (a) and (b) above. We observe that the event $\cG_2$
guarantees that there exists a leftmost interval of the form $\hat
I-k\vec e_1$ which is infected by the Duarte
bootstrap process within time $tj(2N^2-1)$\footnote{The worst case is when sites are infected one by one.}. The event $\cG_1,$ together with the definition
of the Duarte update family $\cU,$ makes sure that the infection of
$\hat I-k\vec e_1$ gets propagated forward to $\hat I-(k-1) \vec e_1,\dots, $ until it
reaches the original interval $\hat I$ in at most $tj (2N^2-1)$
steps. Hence, within time $2jt(2N^2-1)$ the vertex $x_I$ becomes infected and
(a) follows.

It remains to verify (b). The union bound over $k$ gives that for any $\e>0$
\[
\mu\big(\cG_1^c\big)\le jt(1-q)^{\lceil 1/q^3\rceil}\le e^{-\O(1/q^2)}
\quad \text{as } q\to 0,
\]
using \eqref{eq:16} and $j\le N$.

In order to bound from below $\mu\big(\cG_2\big),$ write
\[
\nu:=\min\{\max\{k\in [t/2,t]\colon \text{ the event
  $\cO^1_{\L^{(k)}}(I-k j\vec e_1)$ occurs}\},\,\,\infty\},
\]
and let $\cF=\cap_{i=1}^3 \cF_i$ where, on the event $\{\nu<+\infty\}$: 
{\it 
\begin{enumerate}[{-}]
\item $\quad \cF_1=\{\nu\le t\}$;
\item $\quad \cF_2=\{\forall k\in \big[\lceil 2/q^4\rceil\big]$ the interval $I-\nu j\vec e_1
  +k\vec e_1$ contains an empty vertex$\}$;
\item $\quad \cF_3=\{\exists$ an \emph{upward empty stair} of
  $n=\lceil1/q^3\rceil$ sites belonging to the first
  $\lceil 2/q^4\rceil$ columns of $\cM_t$ immediately to the right of
  $\L^{(\nu)}$, \ie a sequence $(x_1,\dots,x_n)$ of empty sites  of
  the form
  $x_m=(j_m,h_I +m),$ where $h_I$ is
  the height of the uppermost site of $I$ and $\{j_m\}_{m=1}^n$ is a strictly increasing
  sequence$\}$. 
\end{enumerate}
}
We begin by observing that $\cF\subseteq \cG_2$. In fact, $\cF_1$
guarantees the right amount of infection of the last column of
$\L^{(\nu)}$ under healthier boundary condition than those required
by $\cG_2$. $\cF_2$ ensures that such an infection 
propagates over to the first $\lceil 2/q^4\rceil$ columns to the right
of
$\L^{(\nu)}$ while $\cF_3$ guarantees that each time the infection meets an empty
site of the upward stair it grows vertically
by one unit (see Figure \ref{fig:2}). Since the stair contains
$\lceil1/q^3\rceil$ sites, the $\lceil 2/q^4\rceil^{th}$-column of
$\cM_t$ to the right of $\L^{(\nu)}$ contains an infected interval
which is the appropriate horizontal translation of the interval $\hat
I$ and the inclusion $\cF\subseteq \cG_2$ follows. 

Conditionally on $\{\nu = k\}$, the events $\cF_2,\cF_3$  coincide
with two increasing events depending only on sites to
the right of $\L^{(k)}$. Hence, using the FKG inequality,
\begin{gather*}
\mu(\cG_2)\ge \mu(\cF)=\sum_{k\in [t/2,t]}\mu(\nu=k)\mu(\cF_2\cap
\cF_3\tc \nu=k)
\\
\ge \sum_{k\in [t/2,t]}\mu(\nu=k)\mu(\cF_2\tc \nu=k) \mu(\cF_3\tc \nu=k).
\end{gather*}
A union bound gives that, uniformly in $k\in [t/2,t],$
\[
\mu(\cF^c_2\tc \nu=k)\le \lceil 2/q^4\rceil (1-q)^{\ell} \le \lceil 2/q^4\rceil
q^{1/\e}(1+o(1))=o(1), 
\]
if $\e < 1/4$. Using the fact that $X(\o):=\min\{i\ge 1\colon
\o_{(i,+1)}=0\}$ is a geometric random variable of parameter $q$, it
is easy to check that 
\[
\mu(\cF^c_3\tc \nu=k)\le \bbP\Big(\sum_{i=1}^{n}X_i > \lceil 2/q^4\rceil\Big),
\]
where $\{X_i\}_{i=1}^{n}$ are i.i.d copies of $X$. A standard
exponential Markov inequality with $\l=\a q, \a\in (0,1),$ gives
\begin{gather}
\label{eq:LD}
\bbP\Big(\sum_{i=1}^{n}X_i > \lceil 2/q^4\rceil\Big)\le e^{-\l \lceil
  2/q^4\rceil}\Big(\bbE\big(e^{\l X}\big)\Big)^{n}\nonumber\\
\le
\Big(\frac{e^{-2\a}}{(1-\a)(1+o(1))}\Big)^{1/q^3}<\big(1-\a/2\big)^{1/q^3},
\end{gather}
for $\a$ small enough.
In conclusion, if $\e<1/4$, 
\begin{gather*}
\mu(\cG_2)\ge (1-o(1))\mu(\cF_1)\\
\ge (1-o(1))\big(1-\big(1-\mu(\cO^1_{\L_{1,j}}(I))\big)^{t/2}\big)\ge (1-o(1))(1-e^{-4}) 
\end{gather*}
because of \eqref{eq:pippo} and our choice of $t$. That concludes the
proof of property (b). 
\end{proof}
We now turn to the second basic lemma. Recall the definition
\eqref{eq:B2} of the
event $\cB_2.$
\begin{lemma}
\label{lem:D5}  
Choose the basic scales $N,\ell,n_1,n_2$ 
as in
 \eqref{def:ell},\eqref{def:N} and \eqref{eq:8}. Then, for $\e$ small enough,
  \begin{equation}
    \label{eq:19}
\mu\big(\cB_2(n_2-1)\big)\le e^{-\O(1/q^5)},\quad \text{as }
q\to 0.
  \end{equation}
\end{lemma}
\begin{proof}[Proof of Lemma \ref{lem:D5}]
Call $\cH_{i,j}$ the event $\cap_{k=i}^j\{\o\in \O\colon\Phi(\o)_k=\downarrow\}\cap
\cG_{i,j},$ where $\cG_{i,j}$ has been defined in \eqref{eq:B2bis}. 
Clearly 
\[
\mu\big(\cB_2(n_2-1)\big)\le \sumtwo{i,j}{j-i\ge n_2-2}
\mu\big(\cH_{i,j}\big)\le N^2 \maxtwo{i,j\in
  [N]}{j-i\ge n_2-2}% \max_{I\in \cI_j(\ell)}
\mu\big(\cH_{i,j}\big),
\]
and it is enough to prove that 
\begin{equation}
  \label{eq:9}
\maxtwo{i,j\in  [N]}{j-i\ge n_2-2}
\mu\big(\cH_{i,j}\big)\le e^{-\O\big(1/q^5\big)}.
\end{equation}
For this purpose we first describe one important implication of the
event $\cH_{i,j}$.  
\begin{claim}
\label{claim:2}For any $\o\in \cH_{i,j}$ there exists $h\in \bbZ$ satisfying
$|h|\le N^2-(j-1)N + (j-i)\ell,$ such that
\[
C_{h}:=\big(\cup_{k=i}^{j}\{(k-N,h)\}\big)\cap V_{i,j}\subseteq
[Y(\o)\cap V_{i,j}]_{V_{i,j}}^{1,0}.
\] Moreover $C_h$ has length at least $(j-i)(1-o(1))\ge n_2(1-o(1))$
as $q\to 0$.
 \end{claim}
\begin{proof}[Proof of the claim]
Given $\o\in \cH_{i,j}$ let $\G=(x^{(1)},\dots,
x^{(n)})\subseteq [Y(\o)\cap V_{i,j}]_{V_{i,j}}^{1,0}$ be a Duarte
path from $\cC_i$ to $\cC_j$. Since $\Phi(\o)_k=\downarrow$ for all
$k\in \{i,\dots,j\}$ necessarily the cardinality of $\G\cap
\cC_k$ is at most $\ell$ for all $k\in \{i,\dots,j\}$. Therefore the
height $h$ of $x^{(1)}$ satisfies
\[
|h|\le N^2-(j-1)N + (j-i)\ell,
\] 
which, in turn, implies that the corresponding interval
$C_h$ has length greater than the largest
integer $m$ such that  
\[
N^2-(i-1)N-mN\ge  N^2-(j-1)N+(j-i)\ell.
\]
Using that $m+1$ violates the above inequality we get 
\[
m\ge (j-i)(1- \ell/N)-1\ge (1-o(1))n_2.
\] The fact that $C_h\subseteq
[Y(\o)\cap V_{i,j}]_{V_{i,j}}^{1,0}$ follows from Corollary
\ref{cor:duarte-path}.

\end{proof}
It is now easy to finish the proof of the lemma. As in the proof of
Claim \ref{claim:1} and using a union bound over the possible value of the variable $h$ of the
claim, with probability larger than 
\[
1-2N^2e^{-\O(q n_2)}\ge 1-e^{-\O(1/q^5)}, 
\]
every interval  $C_{h}$ as above with $|h|\le N^2-(j-1)N+(j-i)\ell$ meets an
empty upward stair, \ie a sequence $(x_1,\dots,x_{\ell})$ of empty sites
  belonging to the first $n_2/2$ columns crossed by $C_{h}$ and such that
  $x_m=(j_m,h+m)$ with $j_m<j_{m+1}$ for all
    $m\in [\ell]\}.$ If $C_h$ is also infected, then the presence of the
    above empty stair implies that there
  exists $i\le k\le i+\frac 23 n_2$ and a vertical interval
  $I\subseteq \cC_k$ of length at least $\ell$ such that
  $I\subseteq [Y(\o)\cap
  V_{i,j}]_{V_{i,j}}^{1,0}$. The latter property implies that $\Phi(\o)_k=\uparrow$. 
Hence $\mu\big(\cH_{i,j}\big)$ satisfies \eqref{eq:9}
uniformly in $j-i\ge n_2-2$. 
\end{proof}
\subsubsection{Finishing the proof of Proposition \ref{prop:P2}}
\label{sec:finish-proof} Recall
the definition \ref{def:tf} of the test function $\phi$ and of the events $\O_g,\O_{\downarrow}$ and
$\cA_{\epsilon,q}$. Notice that $\O_g\cap \cB_2(n_2-1)^c \subseteq
\cA_{\epsilon,q}$ and that $\O_\downarrow$ is a decreasing
event. Using Lemma \ref{lem:D5} we get
\begin{gather*}
\mu(\phi)\ge \mu\big(\cA_{\epsilon,q}\big)\ge \mu\big(\O_g\cap \cB_2(n_2-1)^c \big)\\
\ge \mu\big(\O_\downarrow\big)\mu\Big(\prod_{|k|\le
  \ell}\o_{(0,k)}=1\Big)- \mu\big(\cB_2(n_2-1)\big)\\
\ge \mu\big(\O_\downarrow\big)(1-q)^{2\ell +1}-e^{-\O(1/q^5)}\ge
q^{O(1)}\mu\big(\O_\downarrow\big) -e^{-\O(1/q^5)},
%\mu(\phi)\ge \mu\big(\cA_{N,\ell,n_1,n_2}\big)\ge 
%\mu\big(\{\omega \in \Omega : (\omega_{V^c}\equiv0,\omega_V)\in\O_g\}\cap \cB_2(n_2-1)^c \big)\\
%\ge \mu\big(\{\omega \in \Omega : 
%(\omega_{V^c}\equiv0,\omega_V)\in\O_\downarrow\}\big)\mu\Big(\prod_{|k|\le
%  \ell}\o_{(0,k)}=1\Big)- \mu\big(\cB_2(n_2-1)\big)\\
%\ge \mu\big(\{\omega \in \Omega : 
%(\omega_{V^c}\equiv0,\omega_V)\in\O_\downarrow\}\big)(1-q)^{2\ell +1}-e^{-\O(1/q^5)} \\
%\ge q^{O(1)}\mu\big(\{\omega \in \Omega : 
%(\omega_{V^c}\equiv0,\omega_V)\in\O_\downarrow\}\big) -e^{-\O(1/q^5)},
\end{gather*}
where in the third  inequality we used the FKG inequality. Using Lemma
\ref{lem:D4} and a union bound, 
\begin{align*}
 \mu\big(\O_\downarrow\big)&\ge 1- \mu\Big(\cup_{j=1}^N\cup_{I\in
   \cI_j(\ell)}\cO^{1,0}_{V_{1,j}}(I)\Big)\\
&\ge 1- 4e^{-(c-5\epsilon)(\log q)^2/q}=1-o(1)
% \mu\big(\{\omega \in \Omega : (\omega_{V^c}\equiv0,\omega_V)\in\O_\downarrow\}\big)&
% \ge 1- \mu\Big(\cup_{j=1}^N\cup_{I\in
%   \cI_j(\ell)}\cO^{1,0}_{V_{1,j}}(I)\Big)\\
%&\ge 1- 4e^{-(c-5\epsilon)(\log q)^2/q}=1-o(1)
\end{align*}
if $\e$ is small enough, where we let $\cI_j(\ell)$ be the family of intervals of the $j^{th}$-column
whose length is at least $2\ell+1$. In conclusion $\mu(\phi)\ge q^{O(1)}$ for
$\e$ small enough.

We now turn to bound from above the Dirichlet form $\cD(\phi)$. By definition,
\begin{gather*}
\cD(\phi)=\sum_{x\in \bbZ^2}\mu\big(c_x \var_x(\phi)\big)=\sum_{x\in
  V}\mu\big(c_x \var_x(\phi)\big)\\
\le \mu(\cA)^{-1}q^{-1}\sum_{x\in V} \mu\big(c_x(\o)\id_{\{\o\in
  \cA\}}\id_{\{\o^x\notin \cA\}}\big)
\end{gather*}
where we used the fact that $\phi$ depends only on 
$\{\o_x\}_{x\in V}$ in the second equality and we wrote $\cA\equiv
\cA_{\epsilon,q}$ for notation convenience. Next we observe that,
\begin{gather}
\sum_{x\in V} \mu\big(c_x(\o)\id_{\{\o\in
  \cA\}}\id_{\{\o^x\notin \cA\}}\big)\nonumber \\\le 
\sum_{x\in V} \mu\big(c_x(\o)\id_{\{\o\in
  \cA\}}\id_{\{\o^x\in \cA^c,\,\o^x\in \cB_2(n_2-1)^c\}}\big)+
\sum_{x\in V} \mu(\id_{\{\o^x\in \cB_2(n_2-1)\}}) \nonumber \\
\le  \sum_{x\in V} \mu\big(c_x(\o)\id_{\{\o\in
  \cA\}}\id_{\{\o^x\in \cA^c,\,\o^x\in \cB_2(n_2-1)^c\}}\big) +
|V|\big((1-q)/q\big)\mu(\cB_2(n_2-1)) \nonumber \\
\le \sum_{x\in V} \mu\big(c_x(\o)\id_{\{\o\in
  \cA\}}\id_{\{\o^x\in \cA^c,\,\o^x\in \cB_2(n_2-1)^c\}}\big)+ e^{-\O(1/q^5)},
\label{eq:21}
\end{gather}
where in the last inequality we used Lemma \ref{lem:D5} and the bound
$|V|\le 2N^3\le e^{O((\log q)^2/q)}$.

Given $x\in V,$ let $\o\in \cA$ be such that $c_x(\o)=1$ and $\o^x\in
\cA^c\cap \cB_2(n_2-1)^c$ and recall that $N_\uparrow(\o)$ counts the number
of up-arrows in $\Phi(\o)$. We claim that $N_\uparrow(\o^x)\ge n_1-1$.
To prove the claim, let $\g$ be a legal path connecting $\O_g$ to
$\o$
such that $\g\cap \cB_i(n_i-1)=\emptyset, \ i=1,2$ and let $\g^x$ be the path
connecting $\O_g$ to $\bar\o^x$ obtained by adding to $\g$ the
transition $\o \to\o^x$. The path $\g^x$ is legal because
$\g$ is legal and $c_x(\o)=1$. Moreover $\g^x\cap \cB_2(n_2-1)=\emptyset$
because $\o^x\notin \cB_2(n_2-1)$. The assumption $\o^x\in \cA^c$ implies
that $\g^x\cap \cB_1(n_1-1)\neq \emptyset$. Using $\g\cap \cB_1(n_1-1)=\emptyset$ the latter
requirement becomes $N_\uparrow(\o^x)\ge n_1-1$ and the claim follows. 

In conclusion, 
\begin{gather*}
\sum_{x\in V} \mu\big(c_x(\o)\id_{\{\o\in
  \cA\}}\id_{\{\o^x\in \cA^c,\,\o^x\in \cB_2^c\}}\big)
\le 
\sum_{x\in V} \mu\big(N_\uparrow(\o^x)\ge n_1-1\big)\\
\le |V|\big((1-q)/q\big)\mu\big(N_\uparrow(\o)\ge n_1-1\big).
\end{gather*}
We finally bound from above $\mu\big(N_\uparrow(\o)\ge n_1-1\big)$ using
Lemma \ref{lem:D4}. Given $n\ge n_1-1$ and $E=\{j_1<\dots <j_n\},\
j_i\in [N],$  let
$\cN_E$ be the event that $\Phi(\o)_j=\uparrow$ if $j\in E$ and
$\Phi(\o)_j=\downarrow$ otherwise. By
construction 
\[
\mu(\cN_E)\le \mu\left(\bigcap_{k=1}^n
\cQ^{1,0}_{V_{j_{k-1}+1,j_{k}}}\right)\le \Big(\max_{i\le j}\mu(\cQ^{1,0}_{V_{i,j}})\Big)^n,
\]
where $j_0:=0$ and 
\[
\cQ^{1,0}_{V_{i,j}}=\{\exists I\in \cI_j(\ell) \text{ such that
} I\subseteq [Y(\o)\cap V_{i,j}]_{V_{i,j}}^{1,0}\}.
\]
where we recall that $\cI_j(\ell)$ is the family of intervals of the $j^{th}$-column
whose length is at least $2\ell+1$.
Lemma \ref{lem:D4} together with a union bound over $I\in \cI_j(\ell)$ give
\begin{gather*}
\max_{i\leq j}\mu\big(\cQ^{1,0}_{V_{i,j}}\big)\le \max_{i\leq j}\sum_{I\in \cI_j(\ell)}\mu\big(I\subseteq
[Y(\o)\cap V_{i,j}]_{V_{i,j}}^{1,0}\big)
\\\le 4 N^4\max_{i\leq j}\max_{I\in \cI_j(\ell)}\mu\big(I\subseteq
[Y(\o)\cap V_{i,j}]_{V_{i,j}}^{1,0}\big)
\le e^{-(c-4\e)(\log q)^2/2q}.
  \end{gather*}
In conclusion, for any $\e$ small enough,
\begin{align*}
\mu\big(N_\uparrow(\o)\ge n_1-1\big)&\le 
\sum_{n=n_1-1}^N \binom{N}{n}
                             e^{-(c-4\e) n(\log q)^2/2q}\\
&\le \sum_{n=n_1-1}^N \Big(N e^{-(c-4\e) (\log q)^2/2q}\Big)^n\\
&\le e^{-\e\,\O((\log q)^4/q^2)},
\end{align*}
because of the choice of $n_1=\e (\log q)^2/2q$. In conclusion, the r.h.s. of
\eqref{eq:21} is smaller than $e^{-\e\O((\log q)^4/q^2)}$ and the proof of
Proposition \ref{prop:P2} is complete.\qed
\section*{Acknowledgment}
We would like to thank R. Morris for several stimulating discussions.

\end{document}